\documentclass{amsart}

\usepackage{amsmath,amssymb,enumerate}
\usepackage{epsfig}

\newtheorem{theorem}{Theorem}

\newtheorem{conjecture}{Conjecture}
\newtheorem{example}{Example}
\newtheorem{proposition}{Proposition}
\newtheorem{lemma}[proposition]{Lemma}
\newtheorem{corollary}[proposition]{Corollary}
\newtheorem{fact}[proposition]{Fact}
\newtheorem{facts}[proposition]{Facts}
\newtheorem{definition}[proposition]{Definition}
\newtheorem{remark}[proposition]{Remark}
\newtheorem{claim}[proposition]{Claim}

\numberwithin{proposition}{section}
\numberwithin{equation}{section}

\newenvironment{demo}{\medskip\noindent{\sc
Proof:}}{\hfill$\square$\medskip}
\newenvironment{demof}[1]{\medskip\noindent{\sc
Proof of #1:}}{\hfill$\square$\medskip}

\newcommand\step[2]{\medskip\noindent{\bf Step #1:} {\it
#2}\medskip}

\newcommand{\new}[1]{{\bf #1}}
\newcommand{\occult}[1]{}

\newcommand\wrt{w.r.t.\ }

\newcommand\bP{P}
\newcommand\diam{{\operatorname{diam}}}
\newcommand\eps{\epsilon}

\newcommand\hT{{\hat{T}}}

\newcommand\hP{{\hat{P}}}
\newcommand\hmu{{\hat{\mu}}}
\newcommand\hS{{\hat{\Sigma}}}

\newcommand\Kappa{{\kappa_{10}}}

\newcommand\leftb{[\hskip-0.14em[}

\newcommand\mult{{\operatorname{mult}}}
\newcommand\NN{{\mathbb N}}
\newcommand\nos{{\flat}}
\newcommand\INT{\operatorname{int}}
\newcommand\perext{{+}}

\newcommand\Proberg{{\mathbb P}_{\operatorname{erg}}\!\!\!}

\newcommand\rightb{]\hskip-0.15em]}
\newcommand\RR{{\mathbb R}}
\newcommand\self{\hookleftarrow}

\newcommand\SxM{\Sigma\ltimes M}
\renewcommand\top{{\operatorname{top}}}
\newcommand\TT{{\mathbb T}}

\newcommand\ZZ{{\mathbb Z}}

\begin{document}

\title[Piecewise Affine Surface Homeomorphisms]{Maximal Entropy Measures for Piecewise Affine Surface
Homeomorphisms}

\date{3/14/2007}

\begin{abstract}
We study the dynamics of piecewise affine surface homeomorphisms
from the point of view of their entropy. Under the assumption of
positive topological entropy, we establish the existence of
finitely many ergodic and invariant probability measures
maximizing entropy and prove a multiplicative lower bound for the
number of periodic points. This is intended as a step towards the
understanding of surface diffeomorphisms. We proceed by building a
jump transformation, using not first returns but carefully
selected ``good'' returns to dispense with Markov partitions. We
control these good returns through some entropy and ergodic
arguments.
\end{abstract}

\author[J. Buzzi]{J\'er\^ome Buzzi}

\address{Centre de Math\'ematiques Laurent Schwartz (UMR 7640)\\
C.N.R.S. \& Ecole polytechnique\\91128 Palaiseau Cedex\\France}

\address{Current Address: Laboratoire de Math\'ematique (UMR 8628)\\C.N.R.S. \& Universit\'e Paris-Sud\\91405 Orsay Cedex\\France}

\email{jerome.buzzi@math.u-psud.fr}

\urladdr{www.jeromebuzzi.com}

\keywords{Dynamical systems; ergodic theory; topological entropy;
variational principle; measures with maximal entropy; Markov
structure; Markov tower; symbolic dynamics; coding.}


\maketitle


\section{Introduction}

Introduced by Anosov and Smale in the 1960's, uniform hyperbolicity is at
the core of dynamical system theory. The corresponding systems
are well-understood since in particular the works of
Sinai, Bowen and Ruelle (see, e.g., \cite{KatokHasselblatt}) and
it is now a central challenge to try
to extend our understanding beyond these systems \cite{BDV}.
We propose that robust entropy conditions provide a way to
do this for new open sets of dynamical systems, by implying a non-uniform but global hyperbolic
structure, especially with
respect to measures maximizing entropy (see section  \ref{sec-defs}
for definitions).

Such invariant probability measures can be thought as describing the
``complexity'' of the dynamics. These measures exist as soon as the dynamics is
compact and $C^\infty$ \cite{Newhouse} or somewhat hyperbolic \cite{BuzziRuette},
though they are known to fail to exist in finite smoothness for interval maps 
\cite{BuzziSIM,RuetteExamples} and diffeomorphisms of four dimensional tori \cite{MisiurewiczExamples}.
Uniqueness problems are usually much more delicate and can be solved
only after a global analysis of the dynamics which we propose to do
under entropy conditions.

\emph{Entropy-expansion} is such a condition. It
requires the topological entropy (see section \ref{sec-defs}) to be strictly larger
than the supremum of the topological entropies of all smooth hypersurfaces. It is robust in the sense
that it is open in the $C^\infty$ topology. Entropy-expanding
$C^\infty$ maps $T:M\to M$ have a finite number of ergodic and invariant probability measures
maximizing the entropy. Their periodic points
satisfy a \new{multiplicative lower bound}:
 \begin{equation}\label{eq:mult-bound}
   \liminf_{n\to\infty,p|n} e^{-nh_\top(T)} \#\{x\in M: T^nx=x\}
    > 0
 \end{equation}
for some integer $p\geq1$ (a period)
(see \cite{BuzziEE,BuzziPQFT} for precise definitions and statements
including other results).

Entropy-expansion is satisfied by plane maps of the type $(x,y)\mapsto
(1.9-x^2+\eps y,1.8-y^2+\eps x)$ for small $\eps$ \cite{BuzziSMF}. On the
interval, the condition reduces to nonzero topological entropy.
In fact, entropy-expansion can be understood as
a generalization of some aspects of one-dimensional dynamics.
Indeed, the previous results were first proved by Hofbauer
\cite{Hofbauer,HofbauerPer} for piecewise monotone maps on the
interval and our approach has built on his techniques.

\medbreak

The techniques used in the above mentioned papers do not apply to
diffeomorphisms (e.g., a diffeomorphism is never entropy-expanding).
However, many properties of interval maps generalize to surface
diffeomorphisms so the following is generally expected:

\begin{conjecture}
  Consider a $C^{1+\eps}$ diffeomorphism ($\eps>0$) of a compact surface
 with nonzero topological entropy.

  The collection of ergodic and invariant probability measures with maximal
  entropy is
  {\bf countable} (possibly finite or empty) and
  the periodic points satisfy a multiplicative lower
  bound if there exists at least one measure with maximum
  entropy.
\end{conjecture}

\begin{conjecture}
   Consider a $C^{\infty}$ diffeomorphism of a compact surface
 with nonzero topological entropy.

   The collection of ergodic and invariant probability measures with maximum
  entropy is
  {\bf finite} and the periodic points satisfy a multiplicative lower
  bound.
\end{conjecture}

By a result of S. Newhouse \cite{Newhouse}, all
$C^\infty$ maps of compact manifolds have at least one
measure of maximum entropy. Also a classical theorem of A. Katok
\cite{KatokIHES} states that, if $T$ is  a $C^{1+\eps}$, $\eps>0$,
diffeomorphism of a compact surface $M$, then the number of periodic
points satisfies a \new{logarithmic lower bound}:
 $$
    \limsup_{n\to\infty} \frac1n
       \#\log\{x\in M:T^nx=x\}
           \geq h_\top(T),
 $$
i.e., a weak version of (\ref{eq:mult-bound}).

\medbreak

This paper presents the proof of the analogue of Conjecture 2 in
the easier case of piecewise affine homeomorphisms. This replaces
distortion of smooth diffeomorphisms by the singularities of the piecewise
affine maps. However this preserves substantial difficulties.
Indeed, there exist piecewise affine
maps on surfaces without a maximum measure (see Appendix \ref{sec-examples},
though the examples I know are discontinuous --or continuous but piecewise
quadratic) or with infinitely many maximum measures (see also Appendix \ref{sec-examples}). Thus this setting, beyond its own interest as a simple and rather
natural class of dynamics, is challenging enough to allow the development
of new tools which we hope will apply to diffeomorphisms.

\subsection{Definitions and Statements}\label{sec-defs}

Let $M$ be a compact two-dimensional manifold possibly with
boundary, {\bf affine} in the following sense. There exists a
distinguished atlas of charts:
 \begin{itemize}
   \item identifying the neighborhood of any point of $M$
with an open subset of $\{(x,y)\in\RR^2:x\geq 0 \;\&\; y\geq0\}$;
  \item inducing affine changes of coordinates.
 \end{itemize}
These charts are called the {\sl affine charts}. The phenomena we
are considering are independent of the global topology, so we could
in fact restrict ourselves to the special cases $M=\TT^2$ or
$M=[0,1]^2$.

A continuous map $T:M\to M$ is said to be \new{piecewise affine}
if there exists a finite partition $P$ of $M$ such that for every
$A\in P$, $A$ and $T(A)$ are contained in affine charts which map
them to polygons of $\RR^2$ with non-empty interiors and $T:A\to
T(A)$ is affine \wrt these affine charts. It is convenient to
replace the partition $P$ by the collection $\tilde P$ of the
interiors of the elements of $P$. Such a partition $\tilde P$ (a
partition up to the boundaries of its elements) is called an
\new{admissible partition} with respect to $T$. We drop the tilde
in the sequel.

Let us recall some facts about entropy (we refer to \cite{WaltersIntro,KellerEntropy} for further
information). The entropy of a non-necessarily invariant subset
$K\subset M$ is a measure of the ``number of orbits'' starting
from $K$. Recall that the \new{$\eps,n$-ball} at $x\in M$ is:
$\{y\in M:\forall k=0,1,\dots, n-1$ $d(T^ky,T^kx)<\eps\}$. The
entropy of $K$ is, according to the Bowen-Dinaburg formula \cite{BowenEntropy}:
 \begin{equation}\label{eq:def-htop}
    h(T,K):=\lim_{\eps\to0} h(T,K,\eps) \text{ with }
    h(T,K,\eps) := \limsup_{n\to\infty} \frac1n\log r(\eps,n,K)
 \end{equation}
where $r(\eps,n,K)$ is the minimal number of $\eps,n$-balls with
union containing $K$. The \new{topological entropy} is
$h_\top(T):=h(T,M)$.

The entropy of an ergodic and invariant probability measure $\mu$
can be defined similarly, according to \cite{KatokIHES}:
 $$
    h(T,\mu):=\lim_{\eps\to0} h(T,\mu,\eps) \text{ with }
    h(T,\mu,\eps):= \limsup_{n\to\infty} \frac1n\log r(\eps,n,\mu)
 $$
where $r(\eps,n,\mu)$ is the minimal number of $\eps,n$-balls
whose union has a $\mu$-measure at least $\lambda$, for a constant
$\lambda\in(0,1)$ ($h(T,\mu)$ is independent of $\lambda$).

The \new{variational principle}  states that for $T:M\to M$ (in
fact for any continuous self-map of a compact metric space \cite{WaltersIntro}):
 \begin{equation}\label{eqVarPrin}
    h_\top(T) = \sup_\mu h(T,\mu)
 \end{equation}
where $\mu$ ranges over the $T$-invariant and ergodic probability
measures.

The following combinatorial expression for topological entropy
follows from observations of S. Newhouse and will be the starting
point of our investigations. 

\begin{proposition}\label{prop-MS} Let $T$ be a piecewise affine
homeomorphism of a compact surface. The topological entropy of $T$
is given by:
 \begin{equation}\label{eq-MS}
   h_\top(T) = \limsup_{n\to\infty} \frac1n\log\#\{A_0\cap f^{-1}A_1\cap\dots
     \cap f^{-n+1}A_{n-1}\ne\emptyset:A_i\in P\}.
 \end{equation}
\end{proposition}

\begin{remark}
The above entropy formula was also obtained by {\sc D. Sands} and {\sc Y. Ishii} \cite{SandsIshii} by
different methods.

Misiurewicz and Szlenk \cite{MS} established the
same formula for piecewise monotone maps of the intervals.
\end{remark}

The proof is given in Sec. \ref{sec-sd}.

\medbreak

The variational principle (\ref{eqVarPrin}) brings to the fore the
ergodic and invariant probability measures $\mu$ such that
$h(T,\mu)=\sup_\nu h(T,\nu)=h_\top(f)$. We call them
\new{maximum measures}. A corollary of the proof of the
previous Proposition is the following existence result (compare with
the examples in appendix \ref{sec-examples}).

\begin{corollary}
A  piecewise affine homeomorphism of a compact surface has at
least one maximum measure.
\end{corollary}

Our main result, finally established in Sec. \ref{sec-main-thm}, is:

\begin{theorem}\label{thm-max-ent}
Let $T:M\to M$ be a piecewise affine homeomorphism of a compact
affine surface. Assume that $h_\top(T)>0$. Then there are finitely
many ergodic, invariant probability measures maximizing the
entropy (or maximum measures).
\end{theorem}

We also obtain as by-products (see Sec. \ref{sec-nbr-per} and
\ref{sec-horseshoes}):

\begin{proposition}\label{prop-periodic}
Let  $T:M\to M$ be a piecewise affine homeomorphism of a compact
affine surface with nonzero topological entropy. 
The periodic points satisfies a {multiplicative lower bound}.
\end{proposition}

\begin{proposition}\label{prop:lsc-entropy}
Let  $T:M\to M$ be a piecewise affine homeomorphism of a compact
affine surface. Let $\mathcal S$ be the singularity locus of $M$,
that is, the set of points $x$ which have no neighborhood on which 
the restriction of $T$ is affine. 

For any $\eps>0$, there is a compact invariant set $K\subset M\setminus\mathcal S$
such that
 $$
   h_\top(f|K)>h_\top(f)-\eps.
 $$
Moreover $f:K\to K$ is topologically conjugate to a subshift of finite
type (see \cite{LM}).
\end{proposition}

\subsection{Outline of the Proof}
We use an alternative approach to the explicit construction of Markov partitions.
We ask less of geometry and use more combinatorics, ergodic theory and entropy estimates 
to accommodate the resulting non-uniqueness of representation. More precisely,
we build small rectangles admitting many returns with "good properties" which allows the
construction of a jump transformation and establish \emph{semi-uniform estimates}, 
that is, uniform estimates holding on subsets of measures that are lower bounded 
\wrt any large entropy measure. The finite number of maximum measures
for the jump transformation follows from results of B.M. Gurevi\v{c} on countable state Markov
shifts. However, the jump transformation is not a first return map to an {\it a priori} 
defined good set. Hence, a careful study of the relation between
the jump transformation and the original dynamics is needed to conclude that the
maximum measures of both systems can be identified. In fact, we analyze
more generally \emph{large entropy measures}, i.e., invariant and ergodic
probability measures with entropy close enough to the supremum.

\medbreak

Let us outline the proof. We start in section \ref{sec-pointwise} by
introducing the natural \emph{symbolic dynamics} of the piecewise affine map
using the partition defined by the singularities of $T$. We first
show that this symbolic dynamics has the same entropy as $T$. This is both a significant
result and a fundamental step in our approach. We then establish that the \emph{(local) stable
$W^s(x)$ and unstable $W^u(x)$ manifolds} of points $x\in M$,
i.e., the sets of points with the same past or future
\wrt the partition $P$, are line segments outside of an entropy-negligible subset.
These line segments can be arbitrarily short and their length may vary discontinuously.
However, we prove semi-uniform lower bounds for their lengths and angles,
using the conditional entropy \wrt the past or the future.
A corollary of these bounds is that the boundary of the partition is negligible
\wrt all ergodic invariant probability measures with nonzero entropy.
\medbreak

At this point, one would like to conclude by an argument of
the following type. If there was a large number of maximum
measures, then one could find two of them, both giving positive measure
to a set $S$ of points with local stable and unstable manifolds
much larger than the diameter of $S$. Hence, one could "jump"
back and forth between typical orbits of each of the two measures.
But one could expect such mixing to allow the construction of
a measure with greater entropy, a contradiction. However,
establishing the increase in entropy seems to require too fine estimates
(of a multiplicative, rather than logarithmic type). We are thus
lead to build a jump transformation with Markov properties which
will reduce the problem to loop counting on a graph, for which these
fine estimates exist (and, indeed, the uniqueness of the maximum
measure has been established in this setting by B.M.~Gurevi\v{c}, see below).

\medbreak

Section \ref{sec-markov-struct} is devoted to building a
Markov structure representing the large entropy dynamics. We first
build arrays of \emph{Markov rectangles} which contain a
significant proportion of the dynamics. These are approximate rectangles
in the sense they contain open subsets of points with local
manifolds that don't cross ("holes" in the local product structure). However we can ensure that the 
relative measure of such points in very small. 
Our techniques require however replacing $T$ by some high power 
$T^L$. 

We then define hyperbolic strips following the geometric picture
of Markov rectangles usual in uniformly hyperbolic dynamics. We
provide tools to build many such strips around typical orbits of
large entropy measures, using visits to the Markov rectangles while
the stable and unstable manifolds are both "rather large".

These hyperbolic strips are Markov in the
sense that they can be freely concatenated as soon as they end and
begin in the same rectangle.
However, to get a useful Markov representation from
this, one needs an invariant
way of "cutting" typical orbits into concatenations of such hyperbolic strips. 
A fundamental difficulty arises here: 
there exist incompatible decompositions, i.e., which do not admit
a common refinement. There does not seem to be an {\it a priori}
natural set the visits to which could be used to defined invariantly
the above "cutting".\footnote{Notice that shadowing lemmas
{\it \`a la} Katok \cite{KatokIHES} give comparable results for surface diffeomorphisms. The
problem is to find invariant decompositions.}

We conclude section \ref{sec-markov-struct} by selecting among hyperbolic strips a
subset of \emph{admissible} ones to get  uniqueness 
in the decomposition of \emph{forward} orbits (this weak uniqueness will
require a more detailed ergodic analysis in section
\ref{sec-lifting}). We obtain a notion of \emph{good return
times} and a Markov structure.

\medbreak

The more technical section \ref{sec-good-returns} relates the good return times to geometric and combinatorial
properties involving the visits to the Markov rectangles and their holes. 
It is shown that large entropy measures cannot have too large average good return time.

\medbreak

Finally, section \ref{sec-lifting} proves the main results. We lift large
entropy measures of $T$ to the jump transformations as finite extensions. 
Using that the latter is isomorphic to a countable state Markov shift, 
a result of Gurevic \cite{Gurevic} allows to conclude the proof of the Theorem.

Proposition \ref{prop-periodic} rests on a classical estimate
of Vere-Jones \cite{VJ} on the number of loops of countable oriented
graphs together with a combinatorial argument to
transfer this estimate to periodic points of $T$.

Proposition \ref{prop:lsc-entropy}, the possibility of approximating in entropy of the whole map
by a compact set away from the singularity set, follows from a similar
property of countable state Markov shifts: they are approximated in entropy
by finite state Markov shifts according to Gurevi\v{c}. 

\medbreak

There are three appendices: (A) recalls some facts about
measure-theoretic entropy, (B) proves a lifting theorem for the
tower defined by a return time and (C) gives some examples of
piecewise affine maps.

\subsection{Some Comments}

The results presented here allow the analysis of the maximal
entropy measures of a simple and natural class of dynamics
by representing them by countable state Markov shifts.
Along the same lines, one can probably make the representation more precise 
to get further results, for instance:
 \begin{itemize}
  \item classification by the topological entropy and periods up to isomorphisms modulo entropy-negligible subsets \cite{BBG};
  \item precise counting of isolated periodic points, e.g., existence of a 
meromorphic extension of the Artin-Mazur zeta function like in \cite{BuzziPQFT};
  \item uniqueness of the maximal entropy measure under a transitivity assumption
and/or a bound on the number of maximum measures in terms of the cardinality of the
partition $P$.
 \end{itemize}
In a slightly different direction, one would like to understand the
nature of the symbolic dynamics of piecewise affine surface homeomorphisms (see Sec. \ref{sec-sd}). Is there
a tractable class of subshifts containing these symbolic dynamics, i.e., a class that would
do for piecewise affine surface homeomorphisms what subshifts of quasi-finite type
\cite{BuzziQFT} do for piecewise monotone interval maps?

It would also be natural to apply the techniques of this paper to more
general dynamics. First, piecewise homographic surface homeomorphisms 
can be analyzed in the same way. Then one could try to analyze other general
classes of piecewise affine maps, especially in higher dimensions 
(e.g., uniformly expanding maps or entropy-expanding maps or 
entropy-hyperbolic homeomorphisms \cite{BuzziICMP}). Most questions are still open
despite some partial results  (see, e.g.,
\cite{BuzziPIM,TsujiiPWA,KR}) and we should stress that new
problems immediately appear. From the point of view of entropy
alone:
 \begin{itemize}
   \item There exist  piecewise affine continuous \emph{maps} on surfaces
and piecewise affine homeomorphisms in \emph{dimension 3} for
which the right hand side of (\ref{eq-MS}) is strictly larger than
the entropy (see Examples \ref{ex-pw-d2-mult}, \ref{ex-pw-d3-mult}
in the Appendix \ref{sec-examples});
  \item Example \ref{ex-pw-d2-max} in Appendix \ref{sec-examples} is a
piecewise affine \emph{discontinuous map} on a surface with no
maximum measure (one can give a continuous, piecewise quadratic
version of it, see example \ref{ex-C0-q-d2-max}). However, I don't
know examples of continuous piecewise affine maps without maximum
measures.
 \end{itemize}

For diffeomorphisms, the main difficulty with our approach is to
find a link between short stable/unstable manifolds and small
entropy, e.g., one would need to relate small Lyapunov charts to
entropy bounds for a smooth diffeomorphism. An analysis of
$C^1$ surface diffeomorphisms with nonzero topological entropy
with dominated splitting is in preparation.

\subsection{Acknowledgements}
I would like to thank the anonymous referees for many remarks which have significantly
improved the exposition.

\section{Pointwise Estimates}\label{sec-pointwise}

\subsection{Symbolic Dynamics}\label{sec-sd}

We define a symbolic dynamics for the map $T$ 
using some admissible partition $P$, that is, a finite collection of disjoint
open polygons with dense union. A key step is showing that $\partial P$ has zero
measure \wrt any $\mu\in\Proberg^0(T)$, where $\Proberg^h(T)$
denotes the set of ergodic, invariant probability measures of $T$
with $h(T,\mu)>h$.

\begin{definition}\label{def:ds}
Let $P$ be an admissible partition.
$x\in M$ is \new{nice} if for every $n\in\ZZ$, $T^n x$ belongs to
an element $A_n$ of the admissible partition $P$. The sequence
$A\in P^\ZZ$ thus defined is the \new{$P$-itinerary} of $x$.

The \new{symbolic dynamics} of $T,P$ is:
 $$
   \Sigma:=\overline{\{A\in P^\ZZ: \exists x\in M\;\forall n\in\ZZ\; T^nx\in
   A_n\}}
 $$
endowed with the shift map: $\sigma(A)=(A_{n+1})_{n\in\ZZ}$.
\end{definition}

A standard result (see, e.g.,
\cite{WaltersIntro}) states that since $\Sigma$ is a subshift, it admits at least one
maximal entropy measure. Hence, a ``close enough'' relation between
the invariant measures of $\Sigma$ and $T$ will imply existence of
a maximum measure also for $T$. By the variational principle, we
shall also get that $T$ and $\Sigma$ have the same topological
entropy. The Misiurewicz-Szlenk formula for $T$ will then follow.
Indeed, for a subshift like $\Sigma$, the topological entropy 
is computed by counting the \emph{cylinders}, 
$[A_0\dots A_{n-1}]:=\{x\in\Sigma:x_0\dots x_{n-1}=A_0\dots A_{n-1}\}$, that is:
 $$
    h_\top(\Sigma)=\lim_{n\to\infty} \frac1n\log
       \#\{[A_0\dots A_{n-1}]\ne\emptyset:A\in P^n\}.
 $$

Neither $T$ nor its symbolic dynamics is an extension of the other in
general, hence it is convenient to introduce the following common
extension:
 $$
    \SxM := \overline{ \{(A,x)\in P^\ZZ\times x:
    \forall n\in\ZZ\; T^nx\in A_n \} }
    \text{ with }
    \hT(A,x)=(\sigma A,Tx).
 $$

The close relation between the measures of $T$ and $\Sigma$
alluded to above is:

\begin{lemma}\label{lem-ext-ent}
Both maps $\pi_1:\SxM\to\Sigma$ and $\pi_2:\SxM\to M$ are entropy
preserving: for every invariant probability measure $\mu$ on
$\SxM$, $h(\sigma,\pi_1\mu)= h(T,\pi_2\mu) =h(\hT,\mu)$. Moreover,
$\pi_1$ and $\pi_2$ induce onto maps between the sets of (ergodic)
invariant probability measures.

In particular, the topological entropies of the three systems are equal 
by the variational principle recalled in eq. (\ref{eqVarPrin}).
\end{lemma}

The proof of the above Lemma rests on two geometric/combinatorial
properties. The first is the following observation by S. Newhouse,
very specific of our setting (it is false in higher dimensions or
without the invertibility assumption, see the Appendix):

\begin{lemma}\label{lem-mult-ent}
The \new{multiplicity entropy} \cite{BuzziAffine}:
 $$
    h_\mult(T) := \limsup_{n\to\infty} \frac1n\log\max_{x\in
    M}\mult(P^n,x) \text{ with }\mult(Q,x):=
    \#\{A\in Q:x\in\overline{Q}\}
 $$
is zero for any piecewise affine homeomorphism of a surface.
\end{lemma}

The second is a property of linear maps:

\begin{lemma}\label{lem-compos-lin}
Let $d\geq1$. For each $n\geq0$, let $T_n:\RR^d\to\RR^d$ be a
linear map. Then
 \begin{multline*}
   \lim_{\eps\to0} \limsup_{n\to\infty} \frac1n\log \max\{ \#S:
   \forall 0\leq k<n\; \diam(T_{k-1}\dots T_1T_0S)\leq 1 \text{ and }\\
   \forall x\ne y\in S\; \exists 0\leq k<n\; \|T_{k-1}\dots
   T_1T_0(x-y)\|>\eps\} = 0.
 \end{multline*}
\end{lemma}

We leave the easy proofs of Lemmas \ref{lem-mult-ent} and
\ref{lem-compos-lin} to the reader.

\begin{demof}{Lemma \ref{lem-ext-ent}}
Lemma \ref{lem-mult-ent}, resp. Lemma \ref{lem-compos-lin},
implies that for all $x\in M$, resp. $\Sigma$, for $i=2$, resp.
$i=1$,
 $$
    h(\hat T,\pi_i^{-1}\{x\})=0
 $$
(this is the entropy of a subset as recalled in (\ref{eq:def-htop})).
Now, $\pi_1:\SxM\to\Sigma$ and $\pi_2:\SxM\to M$ are both compact
topological extensions. Hence, one can apply Bowen's result
\cite{BowenEntropy}:
 $$
 h(\hat T,\hat\mu)= h(\sigma,\pi_1\mu)=h(T,\pi_2\mu)
 $$
for all invariant probability measures $\hat\mu$ of $\SxM$.
\end{demof}

\subsection{Invariant Manifolds and  Lyapunov
Exponents}\label{sec-inv-mfd}

For a fixed partition $P$, we have the following:

\begin{definition}\label{def-mfd-exp}
The \new{stable manifold} at $A\in\Sigma$ is the following set
(convex in the affine charts):
 $$
    W^s(A) := \bigcap_{n\geq1} \overline{T^{-n}A_n}
 $$

The \new{unstable manifolds} $W^u(A)$ are defined by
replacing $n\geq1$ by $n\leq-1$ in the above equation.

The \new{Lyapunov
exponents} along the stable or unstable direction at $A\in\Sigma$ are:
 $$
    \lambda^u(A):= \lim_{n\to\pm\infty}
    \frac1n\log\|{(T_A^n)'}^{\pm1}\|^{\pm1}
    \text{ and }
    \lambda^s(A):=\lim_{n\to\pm\infty}
    \frac1n\log\|{(T_A^n)'}^{\mp1}\|^{\mp1}
 $$
where $T_A^n$ is the affine composition $(T|A_{n-1})\circ
\dots\circ(T|A_0)$ (if $n\geq0$) or
$\left[(T|{A_{-1}})\circ\dots\circ(T|A_{n})\right]^{-1}$ (if
$n<0$).

Any nice $x\in M$ defines a unique itinerary $A$ and we
write $W^s(x)$ for $W^s(A)$, $\lambda^u_+(x)$ for
$\lambda^u_+(A)$ and so on.
\end{definition}

 The first goal of this section is the following
``non-singularity'' result:

\begin{proposition}\label{prop-non-sing}
Let $\mu\in\Proberg^0(T)$. The following holds:
 \begin{itemize}
  \item $\mu(\partial P)=0$ (in particular, $\mu$-a.e. $x\in M$ is nice);
  \item The Lyapunov exponents exist and satisfy $\lambda^s(x)\leq -h(T,\mu)<0<h(T,\mu)\leq\lambda^u(x)$ for $\mu$-a.e. $x\in M$;
  \item $W^s(x)$ and $W^u(x)$ are line segments containing $x$ in
  their relative interiors $\INT W^s(x)$ and $\INT W^u(x)$ for $\mu$-a.e. $x\in M$.
 \end{itemize}
\end{proposition}

\begin{proof} 
Let $\mu\in\Proberg^0(T)$. As we have not yet proved that a.e. $x\in
M$ is nice, we have to work in the extension $\SxM$ to be able to
speak of itineraries, invariant manifolds and so on. By
compactness, there exists an invariant and ergodic probability
measure $\hmu$ of $\hT:\SxM\self$ such that $\pi_2\hmu=\mu$. We
have $h(\hT,\hmu)>0$ by Lemma \ref{lem-ext-ent}.

We first consider the invariant manifolds.

\begin{claim}\label{claim-inv-mfd}
$\hmu$-a.e. $(A,x)\in\SxM$, (i) $W^u(A)$ is a line segment; (ii)
$x$ is not an endpoint of this segment.
\end{claim}

\begin{demof}{the claim}
To begin with, observe that $W^u(\sigma A)\subset T(W^u(A))$ so
that $\dim(W^u(\sigma A))\leq \dim(W^u(A))$. As $\hmu$ is
invariant and ergodic, $\dim(W^u(A))=d$ $\hmu$-a.e. for some $d\in\{0,1,2\}$. Claim (i) above is that $d=1$.

Let $\hP$ be the natural partition of $\SxM$ (coming from the
canonical partition of $\Sigma$). By Lemma
\ref{lem-ext-ent}, $(\hT,\hmu)$ has the same entropy as the corresponding
symbolic system. Thus, $h(\hT,\hmu)=h(\hT,\hmu,\hP)$. Using conditional
entropy (see Appendix \ref{sec-ent-bound}) we can compute $h(\hT,\hmu,\hP)$ as $H_\hmu(\hP|\hP^-)$
where $\hP^-:=\bigvee_{n\geq1} T^{n}\hP$. Observe that $A\mapsto
W^u(A)$ is $\hP^-$-measurable.

We exclude the cases $d=0,2$ by contradiction. Assume first $d=0$,
i.e., $W^u(A)$ is a single point $x\in M$ for $\hmu$-a.e.
$(A,x)\in\SxM$. This implies that:
 $$
    h(\hT,\hmu)=H_\hmu(\hP|\hP^-) = \lim_{n\to\infty} \frac1n
    H_\hmu(\hP^n|\hP^-) \leq h_\mult(T,P) = 0,
 $$
a contradiction, excluding the case $d=0$.

Now assume $d=2$, so $\hmu$-a.e. $\overline{W^u(A)}$ is the closure
of the interior of $W^u(A)$. By construction, two distinct unstable
manifolds have disjoint interiors. Therefore, there can be only
countably many of them, after discarding a set of zero $\hmu$-measure.
In particular, $W^u(A)=W^u(A^{0})$ on a set of positive measure for
some $A^0$. By Poincar\'e recurrence, there exists an integer
$n>0$ such that $T^n(W^u(A^{0}))=W^u(A^{0})$. This implies
that $\pi_1\hmu$ is periodic, hence
$0=h(\sigma,\pi_1\hmu)=h(\hT,\hmu)$. The contradiction proves
(i).

We turn to (ii). If $x\in\partial W^u(A)$, then $T(x)\in\partial
W^u(\sigma(A))$. Thus if (ii) is false, then $x\in\partial W^u(A)$
$\mu$-a.e. But this implies that, for any $\eps>0$, any large $n$,
 $$
 nh(\hat T,\hmu) = H_\hmu(\hP^n|\hP^-)\leq \log 2 +
    \log\max_{x\in M}\#\{A\in P^n:\overline{A}\ni x\}
    \leq \log 2 + (h_\mult(T,P)+\eps)n.
 $$
As $h_\mult(T,P)=0$, it would follow that $h(\hT,\hmu)=0$, a
contradiction.
\end{demof}

We now turn to the exponents. First they do exist by
the classical Oseledets Theorem (see, e.g., \cite{KatokHasselblatt}).

\begin{claim}\label{claim-expo}
For $\pi_1\hmu$-a.e. $A\in\Sigma$, the Lyapunov exponents satisfy:
$\lambda^s(A)<0<\lambda^u(A)$.
\end{claim}

\begin{remark}
The above result will be a consequence of the Ruelle-Margulis
inequality \cite[p. 669]{KatokHasselblatt} \emph{once} we shall have
proved that $\mu(\partial P)=0$.
\end{remark}

\begin{demo}
We establish the existence of a positive Lyapunov exponent
$\mu$-a.e. The existence of a negative exponent will follow by
considering $T^{-1}$. Let $\|\cdot\|_A$ be some measurable family
of norms. Consider the family $\|\cdot\|'_A,A\in\Sigma$ defined by:
 $$
 \|v\|'_A:=\|v\|_A/|W^u(A)|_A \qquad \text{ for }v\parallel W^u(A)
 $$
$|\cdot|_A$ being the length \wrt to $\|\cdot\|_A$ (using the affine
structure). As
$T(W^u(A))\supset W^u(\sigma A)$, we have that
$\|T'|E^u(A)\|'_A\geq 1$ (where $E^u(A)$ is the unstable direction
at $A$ --the \emph{invariant} family of directions defined by
$W^u(A)$) for $\mu$-a.e. $A\in\Sigma$. $T(W^u(A))= W^u(\sigma A)$
$\mu$-a.e. would imply $h(\hat T,\hmu)=H_\hmu(\hat P|\hat P^-)=0$.
Hence, $\|T|E^u(A)\|'_A> 1$ on a set of positive measure and:
 $$
  \lambda^u(A)= \int
    \log\|T'|E^u(B)\|'_B \, d\hmu(B)>0
 $$
for $\mu$-a.e. $A\in\Sigma$.
\end{demo}

We finish the proof of Proposition \ref{prop-non-sing}.

Let $\mu\in\Proberg^0(T)$. Let $\hmu$ be a lift of $\mu$ to
$\SxM$. By Lemma \ref{lem-ext-ent}, $h(\hT,\hmu)=h(T,\mu)>0$.

Claims \ref{claim-inv-mfd} and \ref{claim-expo} prove all the
claims of the Proposition except $\mu(\partial P)=0$.

Now, $W^u(A)$ and $W^s(A)$ are line segments $\mu$-a.e. by Claim
\ref{claim-inv-mfd}. Their  directions carry distinct Lyapunov
exponents by Claim \ref{claim-expo}, hence they must make $\mu$-a.e. a
non-zero angle. If $x\in\partial P$, then $Tx$ or $T^{-1}x$ would be the end
point of at least one of these line segments, a contradiction.
Hence $\mu(\partial P)=0$.
\end{proof}

That $\mu(\partial P)=0$ for all ergodic invariant probability
measures with nonzero entropy has the following immediate but
important consequence:

\begin{corollary}\label{coro-conj-ds}
The partially defined map $\pi:\Sigma'\to M$
$$\{\pi(x)\}:=\bigcap_{n\geq0} \overline{T^k[A_{-k}\dots A_k]}$$
with $\Sigma'$ the subset of $\Sigma$ where the above intersection
is indeed a single point, defines an entropy-preserving
bijection between the sets of ergodic, invariant probability
measures of $T$ and of $\Sigma$ with nonzero entropy.
\end{corollary}

\subsection{Semi-Uniform Estimates}

We obtain now more quantitative estimates, which we call {\bf
semi-uniform} in the sense that they are uniform on a set of
uniformly lower-bounded mass for all large entropy measures. To
state these results, we need the following ``distortion'' estimate.
By compactness of $M$ and invertibility of $T$,
 $$
  d(T):= \sup \left\{ \log\frac{\|T'(x).u\|}{\|T'(x).v\|}:x\in M,\;
   u,v\in\RR^2\setminus\{0\},\; \|u\|=\|v\|\right\}
     < \infty.
 $$

\begin{proposition}\label{prop-size-W}
For any $\mu_0<\frac{h_\top(T)}{d(T)}$, there exist
$h_0<h_\top(T)$, $\theta_0>0$ and $\ell_0>0$ such that for any
$\mu\in\Proberg^{h_0}(T)$, the following properties occur jointly
on a set of measure at least $\mu_0$:
 \begin{gather}\label{eq-lwr-len}
    \rho(x):=\min_{\sigma=s,u}d(x,\partial
   W^\sigma(x))\geq\ell_0
 \\
\label{eq-lwr-angle}
   \alpha(x):=\angle(W^s(x),W^u(x))>\theta_0
 \end{gather}
Here $\angle(W^s(x),W^u(x))$ is the angle between the two lines
defined by $W^s(x)$ and $W^u(x)$. We declare
$\alpha(x)=\rho(x)=0$, if $W^s(x)$ or $W^u(x)$ fail to be line
segments.
\end{proposition}

\begin{remark}
  In fact we obtain a number $\mu_0<1$ arbitrarily close to $1$ satisfying
  (\ref{eq-lwr-len}). However this is not the case for
  (\ref{eq-lwr-angle}). We believe that this cannot be done. 
  Indeed, one can easily build a smooth surface
  diffeomorphism with nonzero entropy such that for some
  $\mu_*>0$ and $h_*>0$, there are invariant probability measures with
  entropy at least $h_*$ such that the stable and unstable directions make an
  arbitrarily small angle on a set of measure at least $\mu_*$. We
  do not know if these measures can be taken to have entropy
  arbitrarily close to the topological entropy or if piecewise
  affine examples exist.
\end{remark}

We first prove the lower bound on angles by comparing the
distortion with the entropy.

\begin{claim}\label{claim-angle}
For any $0<h_1<h_\top(T)$, there exists $\theta_1>0$ such that the
set where $\alpha(x)>\theta_1$ has measure at least $h_1/d(T)$ for
all measures $\mu\in\Proberg^{h_1}$.
\end{claim}

The Ruelle-Margulis inequality applied to $(T,\mu)$ and $(T^{-1},\mu)$
(which is valid as $T'$ is constant on each element of
$P$ and $\mu(\partial P)=0$) yields:
 \begin{equation}\label{eq-ruelle1}
    h_\top(T) \leq \frac{\lambda^u(\mu)-\lambda^s(\mu)}{2}
     =\frac12 \int_M \log
     \frac{\|T'(x)|E^u(x)\|}{\|T'(x)|E^s(x)\|}\, d\mu(x).
 \end{equation}
By continuity there exists $\theta_1>0$ such that, for all
$u,v\in\RR^2\setminus\{0\}$ with $\angle(u,v)\leq\theta_1$,
 $$
 \forall x\in M\setminus\partial P\;\log\frac{\|T'(x).u\|}{\|T'(x).v\|}\leq h_1.
 $$
Therefore, setting $m:=\mu(\{x\in M:\alpha(x)>\theta_1\})$:
 $$
   2h(T,\mu) \leq m \cdot d(T)+(1-m)\cdot h_1
 $$
so that, assuming $h(T,\mu)> h_1$:
 $$
   m\geq \frac{2h(T,\mu)-h_1}{d(T)-h_1}
    \geq \frac{2h(T,\mu)-h_1}{d(T)} > \frac{h_1}{d(T)}.
 $$
This proves Claim \ref{claim-angle}.


\begin{claim}\label{claim-len}
For any $\mu_3<1$, there exists $\ell_0>0$ such that
 \begin{equation}\label{eq-len1}
   \forall\mu\in\Proberg^{h_3}(T)\;
     \mu\left(\{x\in M: d(x,\partial W^u(x))>\ell_0 \}\right) >
   \mu_3
 \end{equation}
for $h_3=h_\top(T)(1-(1-\mu_3)/2)$.
\end{claim}

To prove (\ref{eq-len1}) let $\eps=(1-\mu_3)h_\top(T)/2>0$,
$\Lambda$ be a Lipschitz constant for $T$, $n\geq\log 2/\eps$ be a
large integer and $r=r(\eps,n)>0$ be a small number such that
 $$
 \max_{x\in M} \#\{A\in
P^n:B(x,r)\cap A\ne\emptyset\}\leq \frac12e^{(h_\mult(T,P)+\eps)n}
 \text{ and }
 \# P^n\leq e^{(h_\top(T)+\eps)n}.
  $$
Let $\mu\in\Proberg^{h_3}(T)$, $X_0:=\{x\in M: d(x,\partial
W^u(x))\leq \Lambda^{-n}r\}$ and denote by $\mu|X_0$ the
normalized restriction of $\mu$ to $X_0$. Using standard facts
about entropy (see Appendix \ref{sec-ent-bound}) we get
 \begin{align*}
    n h(T,\mu)&= H_\mu(P^n|P^-)\leq H_\mu(P^n\vee\{X_0,M\setminus X_0\}|P^-)
      \leq H_\mu(\{X_0,M\setminus X_0\}) \\ 
    & \qquad +\mu(X_0)H_{\mu|X_0}(P^n|P^-)
      +(1-\mu(X_0))H_{\mu|M\setminus X_0}(P^n|P^-) \\
    & \leq \log 2 + \mu(X_0)\sup_{x\in X_0}\log\#\{A\in P^n:A\cap
    X_0\cap W^u(x)\ne\emptyset\}\qquad\qquad \\
    & \qquad+(1-\mu(X_0))\log\#P^n \\
    & \leq \log 2 + \mu(X_0)(h_\mult(T,P)+\eps)n +
    (1-\mu(X_0))(h_\top(T)+\eps)n.
 \end{align*}
Hence
 \begin{multline*}
   h(T,\mu) \leq (1-\mu(X_0))h_\top(T)+\mu(X_0)h_\mult(T,P)+\eps
   +\frac1n\log 2
    = \\
     h_\top(T)+2\eps - \mu(X_0)(h_\top(T)-h_\mult(T,P)).
 \end{multline*}
implying that:
 $$
   \mu(X_0) \leq
   \frac{h_\top(T)-h(T,\mu)-\eps}{h_\top(T)-h_\mult(T,P)}
    \leq\frac{h_\top(T)-h(T,\mu)+\eps}{h_\top(T)}
     < 1-\mu_3
 $$
using $h_\mult(T,P)=0$ and $h(T,\mu)>h_3$. The claim is proved.

\bigbreak

\begin{demof}{Proposition \ref{prop-size-W}}
Claim \ref{claim-angle} gives $\theta_0>0$ such that
(\ref{eq-lwr-angle}) holds on a set of measure at least
$h_\top(T)/2d(T)$ wrt all measures in $\Proberg^{h_\top(T)/2}(T)$.
Claim \ref{claim-len} applied to $T$ and $T^{-1}$ with
$\mu_3=1-h_\top(T)/8d(T)$, shows that for
 $$
 h_0=h_\top(T)\left(1-\frac1{16}\frac{h_\top(T)}{d(T)}\right)\geq h_\top(T)/2,$$
(\ref{eq-lwr-angle}) and (\ref{eq-lwr-len}) hold jointly on a set
of measure at least $h_\top(T)/4d(T)$ \wrt all measures in
$\Proberg^{h_0}(T)$.
\end{demof}

\section{Construction of the Markov
Structure}\label{sec-markov-struct}

Roughly speaking, the estimates of the previous section will allow us to build a
collection of (non-uniform) ``Markov rectangles'' which will
``control enough'' of the dynamics to analyze all measures of
large entropy.

\subsection{Markov Rectangles}\label{sec-build-rect}

\begin{definition}
A \new{(Markov) rectangle}
 is a closed topological disk $R$ contained in an affine chart and bounded by four
line segments, alternatively included in stable and unstable
manifolds, making respectively the unstable boundary, $\partial^u
R=\partial^u_1R \cup\partial^u_2R$, and the stable one,
$\partial^s R=\partial^s_1R \cup\partial^s_2R$. See Fig.
\ref{fig-rectangles}.

A \new{Markov array} is a finite collection of Markov rectangles
with disjoint interiors.
\end{definition}

Not every passage of an orbit inside a rectangle is useful. We
need the following properties.

\begin{definition}
A point $x$ is \new{controlled} by a rectangle $R$ if $x$ is nice,
belongs to $R$ and if $W^s(x)$ and $W^u(x)$ each intersects
$\partial R$ in two points. Note that control depends on the
partition $P$ used to define $W^u(x)$ and $W^s(x)$. If necessary
we speak of \new{control with respect} to $P$.

$x\in R$ is
\new{$10$-controlled} if, moreover, $\rho(x)>10\diam R$ where 
$\rho$ was defined in (\ref{eq-lwr-len}). $x\in R$ is
$s$-controlled if $x$ is nice, $x\in R$ and $W^s(x)$ intersects
$\partial R$ in two points.

The sets of controlled (resp. $10$-controlled, $s$-controlled) points is
denoted by $\kappa(R)$ (resp. $\Kappa(R)$, $\kappa_s(R)$).

A point is controlled by a Markov array $\mathcal R$ if it is
controlled by one of the rectangles of the array. We define
$\kappa(\mathcal R)$, $\Kappa(\mathcal R),\kappa_s(\mathcal R)$ in
the obvious way.
\end{definition}

Using the previous lower bounds on the lengths and angles of
invariant manifolds we shall first prove:

\begin{lemma}\label{lem-big-rec1}
There exist numbers $h_0<h_\top(T)$ and $\mu_0>0$ and a Markov
array $\mathcal R$ such that for all $\mu\in\Proberg^{h_0}(T)$,
 $$
   \mu(\Kappa(\mathcal R))>\mu_0.
 $$
\end{lemma}

Our analysis requires the following slightly stronger statement
(i.e., we only tolerate "small holes"):

\begin{lemma}\label{lem-sm-rec0}
There is $\mu_0>0$ such that for any $\eps_0>0$, there exist a
number $h_0<h_\top(T)$ and a Markov array $\mathcal R$ such that
for any $\mu\in\Proberg^{h_0}(T)$,
 \begin{itemize}
   \item $\mu(\Kappa(\mathcal R))>\mu_0$;
   \item
   $\mu(\mathcal R\setminus\Kappa(\mathcal R))<\eps_0\mu_0$.
 \end{itemize}
\end{lemma}

This will be obtained by subdividing the rectangles in the Markov
array from Lemma \ref{lem-big-rec1} into sub-rectangles much
smaller than most stable/unstable manifolds.

\medbreak

The final twist is that as we replace the partition $P$ by the
convex partition $P^{\mathcal R}$ generated by $P$ and the Markov
array $\mathcal R$ (see Fig. \ref{fig-Ptilde}), some invariant
manifolds may shrink, say $\tilde W^u(x):=\bigcap_{n\geq1}
\overline{T^nP^{\mathcal R}(T^{-n}x)} \subsetneq W^u(x)$,
diminishing the set of controlled points. Indeed, $\tilde W^u(f(x))
\subsetneq W^u(f(x))$ when $W^u(x)$ crosses the boundary
of a rectangle from $\mathcal R$ before crossing $\partial P(x)$. 
We shall see however that if
these intersections are sufficiently separated in time, then
$\tilde W^u(x)=W^u(x)$ for most points $x\in\mathcal R$ \wrt
large entropy measures. To guarantee that large separation, we use
the following construction:

\begin{definition}\label{def-perext}
If $\mathcal R$ is an array of Markov rectangles contained in an
element of $P$ and $L$ is a positive integer, then the
\new{$(\mathcal R,L)$-extension} of $(M,T,P,\mathcal R)$ is
$(M_\perext,T_\perext,P_\perext,\mathcal R_\perext)$, defined in
the following way:
 \begin{itemize}
  \item $M_\perext=M\times\{0,\dots,L-1\}$;
  \item $T_\perext(x,k)=(Tx,k+1\mod L)$;
  \item $P_\perext$ is the finite partition of $M_\perext$ which
  coincides with a copy of $P$ on each $M\times\{k\}$ for $k\ne 0$
  and coincides on $M\times\{0\}$ with a copy of $P^{\mathcal R}$;
  \item $\mathcal R_\perext=\{R\times\{0\}:R\in \mathcal R\}$.
 \end{itemize}
\end{definition}

\begin{figure}
\centering
\includegraphics[width=7cm]{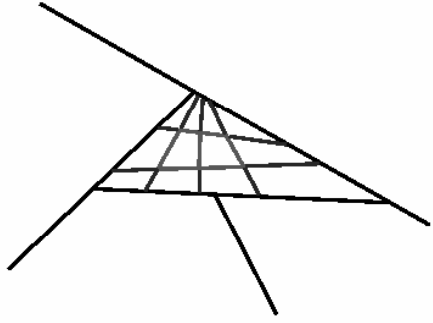}
\caption{The convex partition $\tilde P$ refining both $P$ (outside lines) and $\mathcal R$ (the two
quadrilaterals at the center).}\label{fig-Ptilde}
\end{figure}

The conclusion of this section is:

\begin{proposition}\label{prop-sm-rec}
Let $(M,T,P)$ be a piecewise affine surface homeomorphism with
nonzero entropy. There exist $\mu_0>0$, $h_0<h_\top(T)$ such that
for any $\eps_0>0$, there is a Markov array
$\mathcal R$ and a positive integer $L_0$ with the following
properties. Fix any $L_\perext\geq L_0$ and let
$(M_\perext,T_\perext,P_\perext,\mathcal R_\perext)$ be the
$(\mathcal R,L_\perext)$-extension of $(M,T,P,\mathcal R)$.

For each $\mu\in\Proberg^{h_0}(T)$, there exists an ergodic
invariant probability measure $\mu_\perext$ of $T_\perext$ with
$\pi(\mu_\perext)=\mu$ (where $\pi(x,k)=x$) such that:
 \begin{enumerate}[(i)]
  \item\label{assert-kappa} $L_\perext\cdot \mu_\perext(\Kappa(\mathcal R_\perext))>\mu_0$;
  \item\label{assert-kappa10} $L_\perext\cdot \mu_\perext(\mathcal R_\perext\setminus \Kappa(\mathcal R_\perext))
  <\eps_0\cdot\mu_0$.
 \end{enumerate}
The above controlled sets are defined \wrt to the invariant manifolds
relative to $P_\perext$ (which contains the Markov array $\mathcal R_\perext$).
\end{proposition}

Notice that it is enough to prove our results (Theorem \ref{thm-max-ent} and Propositions
\ref{prop:lsc-entropy} and \ref{prop-periodic}) for some periodic extension.

We now prove Lemmas \ref{lem-big-rec1} and \ref{lem-sm-rec0} and
Proposition \ref{prop-sm-rec}. We begin by the following:

\begin{lemma}\label{lem-big-rec}
Given $\ell_0>0$ and $0<\theta_0<2\pi$, there exists a finite
collection of rectangles $R^{(1)},\dots, R^{(Q)}$ such that:
\begin{enumerate}
  \item $\diam(R^{(i)})<\ell_0/10$;
  \item any $x\in M$ with $\rho(x)>\ell_0$ and $\angle(W^u(x),W^s(x))>\theta_0$ belongs to at least one $R^{(i)}$.
\end{enumerate}
\end{lemma}

This easily implies Lemma \ref{lem-big-rec1} using Proposition
\ref{prop-size-W} and observing that the finite collection of
rectangles above can be subdivided by boundary lines like those of
Fact \ref{fact-bound-W} below so that their interiors become
disjoint, defining the required Markov array $\mathcal R$.


\begin{demof}{Lemma \ref{lem-big-rec}}
Let
 $$
  K_*:=\{x\in M:\rho(x)>\ell_0 \text{ and }
    \angle(W^u(x),W^s(x))>\theta_0\}.
 $$
Let $\{K_j\}_{j=1}^{Q}$ be a finite partition of $K_*$ whose
elements have diameter less than $\theta_0\ell_0/100$  and lie
within an affine chart of $M$. We fix $j$.

Recall that the collection of closed subsets $\mathcal K$ of the
compact metric space $M$ is a compact space \wrt the Hausdorff
metric:
 $$
   d(A,B) = \inf\{ \eps>0: A\subset B(B,\eps) \text{ and  }
   B\subset B(A,\eps) \}
 $$
where $B(A,\eps)$ is the $\eps$-neighborhood of the set $A$, i.e.,
$\{x\in M:d(A,x)<\eps\}$.

The easy proofs of the following two facts are left to the reader.

\begin{fact}\label{fact-limit-W}
Let $A^n\in\Sigma(T,P)$ converge to $A_\perext$. By taking a
subsequence,  $W^s(A^n)$ also converges in the Hausdorff
metric, say to $H\subset M$. Then $H\subset W^s(A_\perext)$.
\end{fact}

\begin{fact}\label{fact-bound-W}
Assume that $K_j$ is as above. Then there exist two points
$x_1,x_2\in\overline{K_j}$, two non-trivial line segments
$L_1,L_2$ and two itineraries $A^1,A^2\in\Sigma(T,P)$ with the
following properties.
 \begin{itemize}
   \item $L_i$ is contained in the boundary of
   $W^s(A^i)$ as a subset of $M$;
   \item in some affine chart, $K_j$ lies between the two lines supporting $L_1$ and
   $L_2$.
 \end{itemize}
We call $L_1,L_2$ a pair of stable boundary lines of $K_j$.
\end{fact}

We now prove Lemma \ref{lem-big-rec}.
Consider two distinct one-dimensional stable manifolds $W^s(A)$ and
$W^s(B)$ which intersect in a single point $p$. $p$ must the
endpoint of at least one of them: otherwise, if $A_n\ne B_n$, then
$p\in\partial A_n\cap\partial B_n$ and both  $W^s(A)$ and $W^s(B)$
are parallel to the same segment of $\partial A_n\cap\partial
B_n$. Thus their intersection contains a non-trivial line segment. 

Observe that if $x,y\in
K_j$, $W^u(x)$ and $W^u(y)$, which are line segments, must have disjoint relative interiors
or be parallel and overlapping. Thus 
 $$
  \angle (W^u(x),W^u(y))\leq \theta_0/50.
 $$
As $\angle(W^u(z),W^s(z))>\theta_0$ for all $z\in K$, we get:
$$\angle(W^u(x),W^s(y))>\theta_0/2.$$

Consider a pair of "stable boundary lines", resp. "unstable
boundary lines", given by Fact \ref{fact-bound-W} applied to
$(T,K_j)$, resp. applied to $(T^{-1},K_j)$. Let $R^{(j)}$ be the
rectangle bounded by these four line segments. $R^{(j)}$ is
contained in the intersection of two strips with almost parallel
sides of width $\leq \diam K_j$ and making an angle at least
$\theta_0/2$. Hence
 $$
   \diam(R^{(j)})<5\diam(K_j)/\theta_0 < \ell_0/20.
 $$
On the other hand, $R^{(j)}\supset K_j$, hence $\bigcup_j
R^{(j)}\supset K_*$.
\end{demof}

\begin{demof}{Lemma \ref{lem-sm-rec0}}
Apply Lemma \ref{lem-big-rec1} to get $\mathcal R$, $\mu_0>0$ and
$h_0<h_\top(T)$. Recall that $\rho(x)$ is the distance between $x$
and the endpoints of its invariant manifolds (or $0$ if one of
those is not a line segment).

By Claim \ref{claim-len} applied with
$\mu_3:=1-\eps_0\mu_0/2$ to $T$ and $T^{-1}$, there exist
$h_1<h_\top(T)$ and $\ell_1=\ell_1(\eps_0\mu_0)$ such that
 $$
  \forall \mu\in\Proberg^{h_1}(T)\; \mu(\{x\in M:\rho(x)<\ell_1\})<\eps_0\mu_0.
  $$
Let us cut each big rectangle $R$ from $\mathcal R$ into
sub-rectangles $R'$ with diameter at most $\ell_1/10$, obtaining a
new Markov array $\mathcal R'$. Using Fact \ref{fact-bound-W}
again, we can do it by finitely many stable and unstable manifolds
(or line segments bounding those). Observe that $\Kappa(\mathcal
R')\supset\Kappa(\mathcal R)$ and that the points in $\mathcal
R'\setminus\Kappa(\mathcal R')$ which have line segments as
invariant manifolds, are $\ell_1$-close to an endpoint of their
stable/unstable manifold. Hence
 $$
   \mu(\kappa_{10}(\mathcal R'))\mu(\kappa_{10}(\mathcal R))\geq\mu_0 \text{ and }
   \mu(\mathcal R'\setminus \Kappa(\mathcal R')) \leq \eps_0\mu_0.
 $$
for all $\mu\in\Proberg^{h_1}(T)$.
\end{demof}

\begin{demof}{Proposition \ref{prop-sm-rec}} We apply Lemma \ref{lem-sm-rec0}
with $\eps_0/2$ obtaining $\mu_0>0$ (independent of $\eps_0$),
$h_0<h_\top(T)$ and a Markov array $\mathcal R$. Let $P^{\mathcal
R}$ be the convex partition previously defined. We go to the
$({\mathcal R},L_\perext)$-extension
$(M_\perext,T_\perext,P_\perext)$ of $(M,T,P)$ for some large
integer $L_\perext$ to be specified. As we observed, there always
exists an ergodic, $T_\perext$-invariant measure $\mu_\perext$
extending $\mu$. Maybe after replacing it by its image under
$(x,i)\mapsto(T^jx,i+j\mod L_\perext)$ for some constant $j\in\{0,1,\dots,L_\perext-1\}$, we get an
ergodic extension $\mu_\perext$ such that
 $$
   L_\perext\cdot\mu_\perext(\Kappa(\mathcal R)\times\{0\})\geq \mu(\Kappa(\mathcal R))\geq\mu_0.
 $$
As the extension is finite-to-one, $\mu_\perext$ has the same
entropy as $\mu$.

Let $x\in M$. If the unstable manifold for $T_\perext$,
$W_\perext^u(x,0):=\bigcap_{n\geq1}
\overline{T_perext^nP_\perext(T_\perext^{-n}(x,0))}$ is strictly shorter
than $W^u(x)\times\{0\}$, then it is bounded by $T_\perext^{kL_\perext}(y,0)$ with $y$
an intersection point of $W^u(T^{-kL_\perext}x)$ for some
$k\geq1$, with one of the new boundary segments, $I$, of 
$P^{\mathcal R}$. Hence $P^{kL_\perext}(T^{-L}x)$ is determined by
the past of $T^{-kL_\perext}x$ and $I$ picked among finitely many
choices. Note that this number of choices depends only on $P$ and 
$\mathcal R$ \emph{but not on $L_\perext$}.

A standard counting argument shows that if this happened on a
subset of $M\times\{0\}$ with $\mu_\perext$-measure at least
$\frac12\eps_0\mu_0\cdot L_\perext^{-1}$, then
 $$
  h(T,\mu)= h(T_\perext,\mu_\perext) \leq (1-\frac12\eps_0\mu_0)h_\top(T)+\eps(L_\perext)
 $$
where $\eps(L)\to0$ as $L\to\infty$. This is strictly less than
$h_\top(T)$ if $L_\perext$ is large enough (which we ensure by
taking $L_0$ large). So it is excluded for large entropy measures.
The $(\mathcal R,L)$-periodic extension
$(M_\perext,T_\perext,P_\perext,\mathcal R_\perext)$ has the
required properties for all large integers $L_\perext$.
\end{demof}

\subsection{Hyperbolic strips}

$(M,T,P)$ is some piecewise affine surface homeomorphism with
nonzero entropy and $\mathcal R$ is some Markov array with
$\mathcal R\subset P$ (eventually $(M,T,P,\mathcal R)$ will be the
previously built periodic extension
$(M_\perext,T_\perext,P_\perext,\mathcal R_\perext)$). We shall
use the following picture to define finite itineraries that can be freely
concatenated. This is adapted from uniformly hyperbolic
dynamics.

\begin{figure}
\centering
\includegraphics[width=13cm]{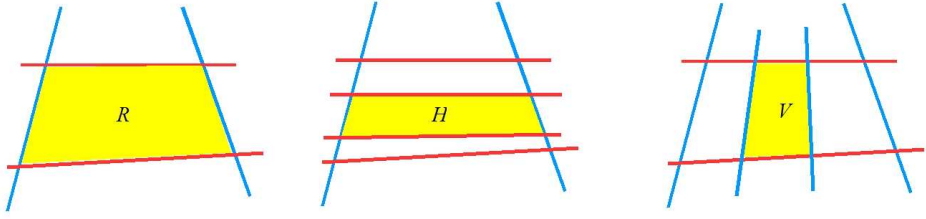}
\caption{From left to right: a rectangle $R$, an $s$-rectangle $H$
and a $u$-rectangle $V$. The (approximately) horizontal,
resp. vertical, line segments are segments of stable, resp.
unstable manifolds.}\label{fig-rectangles}
\end{figure}

\begin{definition}
A quadrilateral $Q$ \new{$u$-crosses} a rectangle $R\in\mathcal R$
if $Q\subset R$ and its boundary is the union of two subsegments
of the stable boundary of $R$ (the stable boundary of $Q$) and
two line segments (the unstable boundary of $Q$), these four segments being
pairwise disjoint, except for their endpoints.
\new{$s$-crossing} is defined similarly.

A \new{$u$-rectangle} is a quadrilateral which
$u$-crosses some rectangle $R\in\mathcal R$ and whose unstable boundary
is made of two segments of unstable manifolds. A
\new{$s$-rectangle} is defined similarly (see Figure
\ref{fig-rectangles}).

For $n\geq1$, a \new{hyperbolic $n$-strip} (or just $n$-strip) is
an $s$-rectangle $S$ such that $\INT T^k(S)$ is included in some
element of $P$ for each $k=0,\dots,n-1$ and $T^n(S)$ is a
$u$-rectangle. A
\new{hyperbolic strip} is an $n$-strip for some $n\geq 1$.
\end{definition}

We write $P_a^b(x)$ for $\bigcap_{k=a}^b
\overline{T^{-(k-a)}P(T^kx)}$ (we assume implicitly that $x$ is
nice --- this fails only on an entropy-negligible set by
Proposition \ref{prop-non-sing}). The following is immediate.

\begin{facts}$ $

1. A hyperbolic $n$-strip is necessarily of the form
$\bP_0^n(x)$ for some $x\in R$.

2. Two hyperbolic strips are either nested or have disjoint
interiors.
\end{facts}

We now give some tools to build hyperbolic strips.

\begin{lemma}\label{lem-strip-2}
For $0<m<n$, if $\bP_0^m(x)$ and $\bP^n_m(x)$ are both hyperbolic
strips, then so is $\bP^n_0(x)$.
\end{lemma}

This is easy to show using Fig. \ref{fig-composition}.
Sufficiently long invariant manifolds allow the construction of
hyperbolic strips from scratch:

\begin{figure}
\centering
\includegraphics[width=13cm]{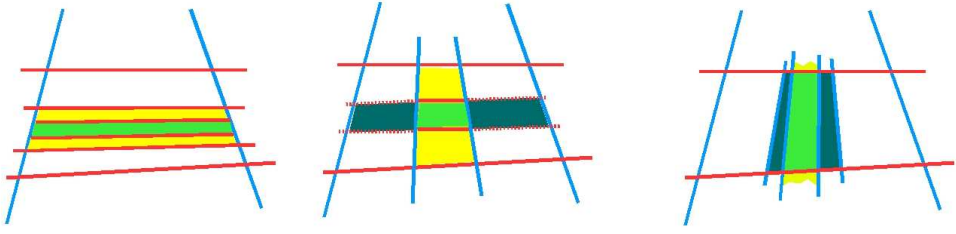}
\caption{Proof of Lemma \ref{lem-strip-2}. From left to right: (a)
$\bP_0^n(x)\subset\bP_0^m(x)\subset R$; (b) the $u$-rectangle
$T^m(\bP_0^m(x))$ crossing the $s$-rectangle $\bP_m^n(x)$ (both
inside $R'$); (c) $T^n\bP_0^n(x)\subset T^m\bP_m^n(x)\subset R''$
($R,R',R''\in\mathcal R$).}\label{fig-composition}
\end{figure}

\begin{lemma}\label{lem-strip-10}
Let $x\in\Kappa(\mathcal R)$ and $n\geq1$ such that
$T^nx\in\Kappa(\mathcal R)$. Then $\bP^{n}_0(x)$ is a hyperbolic
strip.
\end{lemma}

Observe that the weaker condition $x\in\kappa(\mathcal R)\cap
T^{-n}\kappa(\mathcal R)$ does not imply that $\bP_0^n(x)$ is a
hyperbolic strip.

\begin{demo}
See Fig. \ref{fig-strip10}. Let $R,R'$ be the elements of
$\mathcal R$ containing $x$ and $T^nx$. Consider the diamond (the quadrilateral) 
$L$ generated by $W^s(x)$ and $T^{-n}W^u(T^nx)$. By convexity, $L$
is contained in $T^{-1}\bP^n_1(x)$. 

Consider one side $[uv]$ of $T^nL$ with $u\notin R'$ and $d(x,u)>10\cdot
\diam(R')$. Let $\{a\}:=[uv]\cap\partial^s R'$. We have $d(u,v)\geq
d(u,T^nx)-\diam(R')\geq 9 \diam(R')$. Hence $d(v,a)\leq\diam(R')\leq (1/9) d(u,v)$.

Let $(abcd)$ be the quadrilateral defined by the points of 
$\partial T^nL\cap\partial R'$. It $u$-crosses $R'$. It
remains to prove that $L$ $s$-crosses $R$ so that
$R\cap T^{-n}Q$ is the desired hyperbolic strip and must be
$\bP_0^n(x)$. 

It is enough to check that $T^{-n}\{a,b,c,d\}$ 
lies outside of $R$. 
The map being piecewise affine, $d(T^{-n}v,T^{-n}a)\leq (1/9) d(T^{-n}u,
T^{-n}v)$. But $d(x,T^{-n}v)\geq 10\cdot\diam(R)$. Hence,
$T^{-n}a$ is outside of $R$. The same holds for the pre-images
of $b,c,d$. Hence $T^{-n}(abcd)$ $s$-crosses $R$. Therefore
\end{demo}

\begin{figure}
\centering
\includegraphics[width=13cm]{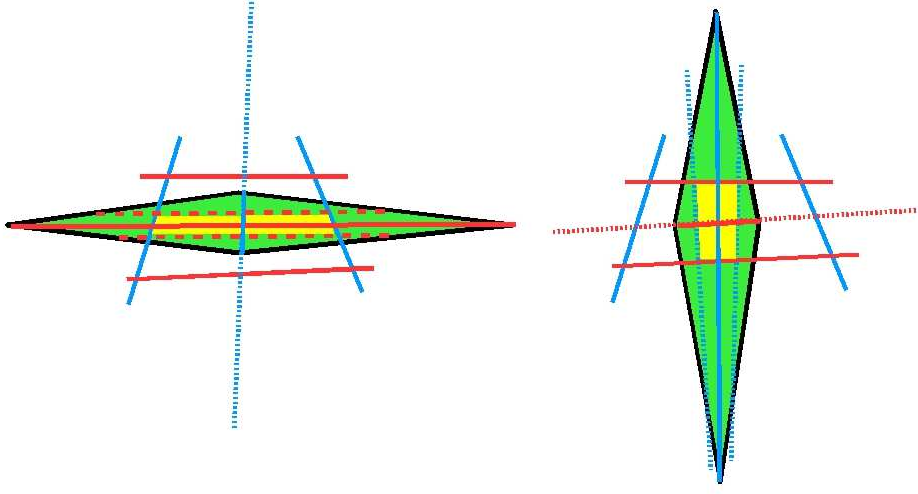}
\caption{Construction of the hyperbolic strip in the proof of
Lemma \ref{lem-strip-10} (left: time $0$ around $R$, right: time
$n$ around $R'$; approximately (vertical) horizontal lines are
segments of (un)stable manifolds; dashed ones are (pre)images of
regular ones. The colored ``diamond'' is $L$. The rectangle
inscribed in $L$ is the hyperbolic strip.}\label{fig-strip10}
\end{figure}

\begin{corollary} \label{cor-strip-3}
Let $n\geq1$ be such that $\bP^n_0(x)$ is a hyperbolic strip and
$T^nx\in\Kappa(\mathcal R)$. If $m>n$ satisfies $T^mx\in\Kappa(\mathcal R)$, then
$\bP_0^m(x)$ is also a hyperbolic strip.
\end{corollary}

\begin{demo}
By Lemma \ref{lem-strip-10}, $\bP_n^m(x)$ is a hyperbolic strip.
Apply Lemma \ref{lem-strip-2} to conclude.
\end{demo}

We need the following technical fact.

\begin{lemma}\label{lem-struts}
Let $\mu$ be an atomless invariant probability measure. For $\mu$-a.e.
$x\in\kappa_s(R)$, for some $R\in\mathcal R$ and all $n\geq1$, the intersection of
$W^s(x)$ with $\partial^u R$ is disjoint from all the vertices
of $P^n$, $n\geq1$. In particular, $\partial^u R\cap
\partial P^n(x)$ is the union of two non-trivial segments.
\end{lemma}

\begin{demo}
We proceed by contradiction assuming that the above fails: on a
subset of $\kappa_s(R_i)$ with positive measure at least one of
these intersection points coincides with a vertex $z$ of the
polygon $\partial P^n(x)$ (so $W^s(z)=W^s(x)$).  Reducing this
subset, we assume the vertex $z$ to be a fixed one, say
$z_\perext$.

By Poincar\'e recurrence, there must exist infinitely many
$n\geq0$ such that $T^nx\in W^s(z_\perext)$. Considering two such
integers $n_1<n_2$, we get that
$T^{n_2-n_1}(W^s(z_\perext))\subset W^s(z_\perext)$. This implies
that all points of $W^s(z_\perext)$ converge to a periodic orbit.
Thus, the ergodic decomposition of $\mu$ has an atom, a
contradiction.
\end{demo}

We show that if $x\in\kappa_s(\mathcal R)$, then subsequent visits to
$\Kappa(\mathcal R)$ either give a hyperbolic strip or a shadowing property
which will lead to an entropy bound.

\begin{lemma}\label{lem-strip}
Let $x\in\kappa_s(\mathcal R)$ and $0\leq m< n$ be such that
$T^mx,T^nx\in\Kappa(\mathcal R)$. Excluding a set of zero measure of points
$x$, if $\bP^{n}_0(x)$ is not a hyperbolic strip then
$\bP_0^{m}(x)$ determines $\bP^{n}_0(x)$ up to a choice of
multiplicity $4$.
\end{lemma}

\begin{demo}
$W^s(x)$ crosses the rectangle $R\in\mathcal R$ containing $x$. 
Hence Lemma \ref{lem-struts} implies that
$\partial P_0^{m}(x)\cap\partial^uR$ is the union of two
unstable, non-trivial segments: $[a,b]$, $[c,d]$. Let $[a',b']$,
$[c',d']$ be their images by $T^m|\bP_0^{m}(x)$. Let $Q'$ be the
quadrilateral generated by them. By convexity $Q'\subset
T^{m}(\bP^{m}_0(x))$.

$H:=\bP_{m}^{n}(x)$ is a hyperbolic strip by Lemma
\ref{lem-strip-10}. $Q'$ and $H$ intersect. If $\INT
H\cap\{a',b',c',d'\}=\emptyset$, then $Q'$ would go across $H$,
and $P_0^m(x)$ would be a hyperbolic strip, contrary to
assumption.

Thus, at least one of the four vertices $a,b,c,d$ determined by
$P_0^m(x)$ is contained in $\INT H$, this point determines $H$ and
therefore $P_0^n(x)$ as claimed.
\end{demo}

\subsection{Admissible Strips and Good Returns}

In this section, $\mathcal R$ is some Markov array with $\mathcal
R\subset P$. Hyperbolic strips defined above have no uniqueness
property: a point $x\in\kappa_s(\mathcal R)$ sits in an infinite
sequence of nested hyperbolic strips. This motivates the following
notion.

\begin{definition}
For $n\geq1$, the \new{admissible $n$-strips} are defined by
induction on $n$. A $1$-strip is always admissible. For $n>1$, an
admissible $n$-strip $S$ is an $n$-strip such that for all $1\leq
m<n$ such that $S$ is included in an admissible $m$-strip,
$T^m(S)$ meets the interior of no hyperbolic strip. An \new{admissible strip} is
an admissible $n$-strip for some $n\geq1$.
\end{definition}

Figure \ref{fig-strips2} shows a hyperbolic
$n$-strip $S$ (hatched) which is not admissible.

\begin{figure}
\centering
\includegraphics[width=10cm]{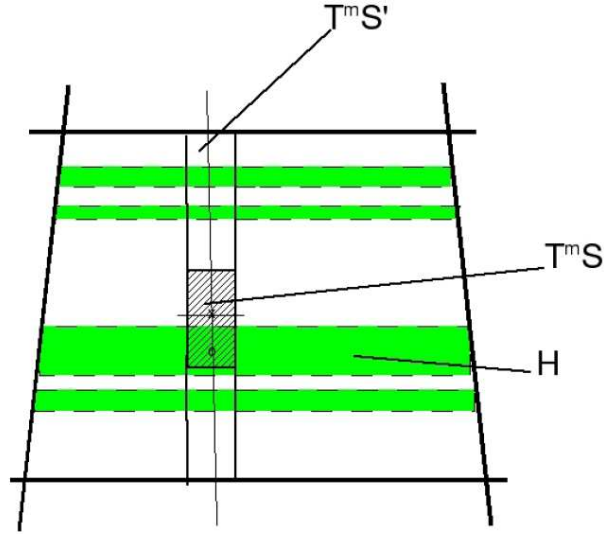}
\caption{An example of a non-admissible $n$-strip $S$. For some $0<m<n$,
$S$ is contained in some $m$-strip $S'$ (hence $T^mS'$ $u$-crosses) and $T^mS$ meets the
interior of some hyperbolic strip $H$. The above represents $T^mS\subset T^mS'$ and
$H$. Point $x$ is at the crossing of the stable and unstable lines and point $o$ is 
a little below, on the same unstable line. The four $s$-crossing rectangles are the
maximum hyperbolic strips. 
In such a situation, it might be possible to then split the itinerary
of $S$ into that of $S'$ followed by that of $H$, yielding a choice 
for representing the itinerary of any point like $o$ above, at each 
of its visit to $S$. Compare with Lemma \ref{lem-adm-uniq}.}\label{fig-strips2}
\end{figure}

\begin{definition}\label{def-good-return-time}
Let $(M,T,P)$ be a piecewise affine surface homeomorphism
with a Markov array contained in $P$.

For a point $x\in M$,  the \new{(good) return time} is
$\tau=\tau(x)$, the minimal integer $\tau\geq1$ such that both
following conditions hold:
 \begin{itemize}
  \item $x$ belongs to an admissible $\tau$-strip;
  \item $T^\tau(x)\in \kappa_s(\mathcal R)$.
 \end{itemize}
These conditions are defined \emph{with respect to} some partition
containing some Markov array.

If there is no such integer $\tau$, then we set $\tau(x)=\infty$.
\end{definition}

\begin{remark}
Note that, at this point, we break the symmetry between the future
and the past.
\end{remark}

We shall use repeatedly the following obvious observation:

\begin{fact}\label{fact-first-hyp-adm}
If $n$ is the smallest integer such that $\bP_0^n(x)$ is a
hyperbolic strip (equivalently: $\bP_0^n(x)$  is an $n$-strip which
is not contained in a $k$-strip for any $k<n$; $\bP_0^n(x)$  does
not meet a $k$-strip distinct from $\bP_0^k(x)$ for any $k<n$; $\bP_0^n(x)$  is a hyperbolic
strip which is maximum \wrt inclusion) then $\bP_0^n(x)$ is an
admissible $n$-strip.
\end{fact}

\begin{remark}
One could consider the following changes in the definition of
admissibility:

1) replacing ``$T^m(S)$ meets no hyperbolic strip'' by ``$T^m(S)$
meets no admissible strip'' would not change the notion. Indeed,
suppose that $T^m(S)$ meets a hyperbolic strip $H$. Let $k\geq1$
be the smallest integer such that $H$ is contained in a $k$-strip,
say $H_k$. The minimality of $k$ implies that $H_k$ is admissible
and $H_k\supset H$ so that $T^m(S)$ meets $H_k$.

ii) replacing ``$S$ is included in an admissible $m$-strip'' by
``$S$ is included in a \emph{hyperbolic} $m$-strip'' would exclude
some admissible strips and so would cause a problem in the proof
of the (key) Claim \ref{claim-det} (for the proof that $k=n_i$ in
the notations there).
\end{remark}

Admissibility gives the following uniqueness property. 
Denote the one-sided symbolic
dynamics by: $\Sigma_+(T,P)=\{A_0A_1A_2\dots\in P^\NN:
A\in\Sigma(T,P)\}$.

\begin{lemma}\label{lem-adm-uniq}
A positive itinerary $A\in\Sigma_+(T,P)$ can be decomposed in at
most one way as an infinite concatenation of admissible strips.
\end{lemma}

\begin{demo}
Consider two distinct decompositions of $A$ into admissible
strips, that is, $n_0=0<n_1<n_2<\dots$ and $m_0=0<m_1<m_2<\dots$,
such that $A_{n_i}\dots A_{n_{i+1}}$ and $A_{m_i}\dots
A_{m_{i+1}}$ are admissible strips for all $i\geq0$. By deleting
the identical initial segments, we can assume that the
decompositions differ from the beginning, say $n_1<m_1$. It
follows that the admissible $m_1$-strip $H:=\overline{[A_0\dots
A_{m_1}]}$ is contained into the $n_1$-admissible strip
$\overline{[A_0\dots A_{n_1}]}$. Thus $T^{n_1}(H)$ meets
$\overline{[A_{n_1}\dots A_{n_2}]}$ which is another admissible
strip, contradicting admissibility.
\end{demo}

\section{Analysis of Large Return Times}\label{sec-good-returns}

In this section $(M,T,P)$ is a piecewise affine homeomorphism with
positive topological entropy. In the first two subsections, we analyze
the implications of a long return time $\tau(x)(x)$ from a geometric
and then a combinatorial point of view. We then apply this to
invariant measures with very large "average" return times to
bound the entropy of these measures.

\subsection{Geometric Analysis}
We analyze geometrically the implications of a large return time.

\begin{proposition}\label{prop-adm-strip} 
Let $(M,T,P)$ be a piecewise affine surface homeomorphism and let
$\mathcal R\subset P$ be a Markov array. Let $x\in\kappa_s(\mathcal R)$
and let $1\leq N\leq\tau(x)$.

Let $0\leq N_1\leq N_2\leq N_0<N$ be defined as follows:
 \begin{itemize}
  \item $0<N_0<N$ be the smallest integer such that
$T^{N_0}x\in\Kappa(\mathcal R)$ and $\bP_0^{N_0}(x)$ is a
hyperbolic strip (we set $N_0:=N$ if there is no such integer);
  \item $0\leq N_1<N$ be the smallest integer such that
$T^{N_1}x\in\Kappa(\mathcal R)$ (we set $N_1:=N_0=N$ if there is
no such integer);
 \item $0\leq N_2<N_0$ be the largest integer such that
$T^{N_2}x\in\Kappa(\mathcal R)$ (we set $N_2:=N_1=N_0$ if there is
no such integer).
 \item$n_1,\dots,n_r$ ($r\geq0$) be  the \emph{admissible times}, that is
 the successive integers in:
 $$
  \{0\leq k<N:\bP_0^{k}(x)\text{ is an admissible
strip }\}
 $$
 with the convention $n_{r+1}=N$;
 \item $m_{i1},\dots,m_{is(i)}$
 ($s(i)\geq0$) be the \emph{hyperbolic times}, that is, for each $i$,
 the successive integers:
 $$
\{n_i<m<n_{i+1}: \bP_0^{m}(x)\text{ is a $m$-strip and
}T^mx\in\kappa_s(\mathcal R)\}.
 $$
  with the convention  $m_{is(i)}:=m_{i1}:=n_{i+1}$ and $s(i)=0$ if
the above set is empty.
 \end{itemize}

Then $P^N(x)$ is determined, up to a choice of multiplicity
$4\cdot 2^r$, by:
 \begin{enumerate}
  \item the integers $N_1,N_2,r$ and $n_i,m_{i1},m_{is(i)}$ for $1\leq
  i\leq r$;
  \item $P(T^kx)$ for $k\in\leftb 0,N_1\rightb\cup\leftb N_2,n_1\leftb$;
  \item $P(T^kx)$ for
  $k\in\bigcup_{i=1}^r\leftb n_i,m_{i1}\rightb\cup\leftb m_{is(i)},n_{i+1}\rightb
  \setminus\leftb0,N_1\rightb$.
 \end{enumerate}
($\leftb a,b \leftb$ denotes the integer interval with $a$
included and $b$ excluded, etc.).
\end{proposition}

\begin{demo}
The inequality $0\leq N_1\leq N_2\leq N_0$ is easily checked.

We can assume that $N_1<N$, as otherwise there is nothing to show.

We first claim that $\bP_0^{N_1}(x)$ determines $\bP_0^{N_2}(x)$ up to a
choice of multiplicity $4$.

If $N_1\geq N_0$, then $N_1=N_2=N_0$ and there is nothing to show.
Otherwise $N_1\leq N_2 < N_0$ and $T^{N_1}x,T^{N_2}x\in\Kappa(\mathcal R)$.
As $\bP_0^{N_2}(x)$ is not hyperbolic this implies, by Lemma \ref{lem-strip},
the above claim.

It remains to prove the following claim.
\end{demo}

\begin{claim} \label{claim-det}
Except for an entropy-negligible subset of points $x\in M$, the following holds. Given some $1\leq
i\leq r$ with $s(i)>0$, $Q:=P^{m_{i1}}_{n_i}(x)$ and integers
$n_{i},m_{i1},m_{is}$, there are only two possibilities for
$P^{m_{is}}_{n_i}(x)$ ($s$ denotes $s(i)$).
\end{claim}

\begin{demo}
Let $R,R'\in\mathcal R$ be the rectangles containing $T^{n_i}x,
T^{m_{i1}}x$ and let $\ell$ be the line segment through
$T^{n_i}x$, directed by $W^s(T^{n_i}x)$ and bounded by $\partial
R$. We first show that $\ell\not\subset Q$.

$T^{n_i}x\notin\kappa_s(R)$ as $\tau(x)>n_i$. Thus
$W^s(T^{n_i}x)$ does not $s$-cross $R$: $\ell\not\subset
W^s(T^{n_i}x)$. There exists $k>n_i$ such that
$T^{k-n_i}\ell$ is not contained in the closure of an element of
$P$. Take $k\geq1$ minimal. If one had $k>m_{i1}$, then $T^{m_{i1}-n_i}\ell\subset
W^s(T^{m_{i1}}x)$ (recall that $T^{m_{i1}}x\in\kappa_s(R')$) so that for all $k\geq m_{i1}$,
$T^{k-n_i}\ell$ would be contained in an element of $P$, implying $\ell\subset W^s(T^{n_i}x)$, a
contradiction. Thus, $k\leq m_{i1}$ and $\ell\not\subset Q$ as
claimed.

Disregarding an entropy-negligible set of points $x$, we can assume
that $\ell$ divides $Q$ into two subsets with non-empty interiors,
say $Q_+,Q_-$. There cannot exist stable manifolds
that $s$-cross $R$ both in $Q_+$ and $Q_-$: by convexity this would
imply that $W^s(T^{n_i}x)$ also $s$-crosses. Thus there is an $s$-rectangle
in $R$ disjoint from $\kappa_s(R)$ which contains at least one of
$Q_+,Q_-$. Let $W^s(B_+)$ ("above") and $W^s(B_-)$ be the stable manifolds
bounding this gap (recall Fact \ref{fact-limit-W}). Also
at least one of $W^s(B_\pm)$ (say $W^s(B_+)$) is not contained in
$Q$ so the interior of $Q$ does not meet $W^s(B_+)$. Thus $Q$
determines $W^s(B_+)$: it is the "lowermost" stable manifold
"above" $Q$ which crosses $R$. Likewise $W^s(B^-)$
is the "uppermost" stable manifold "below" $W^s(B^+)$ which
crosses $R$.

By definition, $\mathcal S:=\bP_0^{m_{is}}(x)$ is hyperbolic. Also
$m_{is}\in\rightb n_i,n_{i+1}\leftb$ is not admissible hence there
exists $0<k<m_{is}<n_{i+1}$ such that $\mathcal
S$ is included in an admissible $k$-strip and $T^{k}(\mathcal S)$
meets an admissible strip. If $k<n_i$, then the same admissible
strip would preclude the admissibility of $\bP_0^{n_i}(x)$ is 
admissible. The contradiction shows that $k\geq n_i$. As $n_{i+1}$ 
is the smallest admissible time after $n_i$, $k=n_i$.

Thus $T^{n_i}(\mathcal S)$ meets an admissible strip which 
must be either "above" $W^s(B_+)$ or
"below" $W^s(B_-)$. This implies that $\bP^{m_{is}}_{n_i}(x)$ meets
and therefore contains $W^s(B_\pm)\cap R$ (for one of the signs $\pm$).
It follows that $Q$ determines $\bP^{m_{is}}_{n_i}(x)$, up 
to a binary choice.
\end{demo}

\subsection{Combinatorial Estimates}

\begin{remark}
In the remainder of this section, the controlled sets and return times are understood 
to be with respect to the partition $P_\perext$ and the
Markov array $\mathcal R_\perext\subset P_\perext$ of some periodic
extension as defined in Definition \ref{def-perext}.
\end{remark}

We extract from Proposition \ref{prop-adm-strip} the following
complexity bound.

\begin{proposition}\label{prop-ent}
Let $(M,T,P)$ be a piecewise affine surface homeomorphism and
let $\mathcal R$ be a Markov array. Let $\eps_*>0$ and let 
$C_*=C_*(\eps_*)<\infty$ be such that
 $$
   \forall n\geq 0\;\;  \#(P^{\mathcal R})^n\leq C_* e^{(h_\top(T)+\eps_*)n}
 $$
For each positive integer $L$, let $(M_\perext,T_\perext,P_\perext)$
be the $(L,\mathcal R)$-extension of $(M,T,P)$.

Let $L,N\geq1$, $M,R,S\geq0$ be some integers. Consider
$\mathcal I=\mathcal I(N,L,M,R,S)$ the set of cylinders $(P_\perext)^N(x)$ for
$x\in\kappa_s(\mathcal R_\perext)$ such that, in the
notations of Proposition \ref{prop-adm-strip} applied to the periodic
$(L,\mathcal R)$-extension:
 \begin{itemize}
  \item $\tau(x)\geq N$;
  \item $r=R$ and $\#\{m_{ij}\geq N_0:1\leq i\leq r,1\leq j\leq s(i)\}=S$;
  \item $N_2-N_1=M$.
 \end{itemize}
Let $\rho>R/N$. Then:
 $$
   \log\#\mathcal  I\leq
       (h_\top(T)+\eps_*)(N-L(S-R)-M+S)+K_*(\rho,N)N+(\rho+3/N)N\log C_*
 $$
where $K_*(\cdot)$ and $K_*(\cdot,\cdot)$ are universal\footnote{$K_*$ does not depend on any
of the data $T:M\to M,N,M,R,S,L,C_*,\eps_*$.} functions satisfying
$K_*(\rho,N)\downarrow K_*(\rho)$ when $N\to\infty$ and
$K_*(\rho)\downarrow 0$ when $\rho\to0$.
\end{proposition}

The proof of the above will use:

\begin{lemma}\label{lem-mij}
In the notation of Proposition \ref{prop-adm-strip},
 \begin{enumerate}
 \item $n_1$ is the smallest integer such that
$(\bP_\perext)_0^{n_1}(x)$ is hyperbolic and $n_1\leq N_0$;
  \item $\{n_{i}:1\leq i\leq r\}
    \subset \{0\leq k<N:T_\perext^kx\in \mathcal R_\perext\setminus \kappa(\mathcal R_\perext)\}$;
  \item $\{N_0\leq k<N:T_\perext^kx\in\Kappa(\mathcal R_\perext)\} \subset
   \{m_{ij}:1\leq i\leq r\text{ and }1\leq j\leq
  s(i)\}$.
\end{enumerate}
\end{lemma}

\begin{demo}
By definition $n_1$ is the smallest integer such that
$(\bP_\perext)_0^{n_1}(x)$ is an admissible strip so (1) is just Fact
\ref{fact-first-hyp-adm}.

$(\bP_\perext)_0^{n_i}(x)$ being an admissible strip,
$T_\perext^{n_i}x\in\mathcal R_\perext$. $n_i<N$ so
$T_\perext^{n_i}x\notin\kappa(\mathcal R_\perext)$, proving (2).

The $m_{ij}$ are the times $m\in\rightb n_1,N\leftb$ (or,
equivalently, $m\in\leftb 0,N\leftb$ by property (1)) such that
$T_\perext^mx\in\kappa_s(\mathcal R_\perext)$ and
$(\bP_\perext)_0^m(x)$ is a hyperbolic, but not admissible, strip.
As $(\bP_\perext)_0^{N_0}(x)$ is a hyperbolic strip and
$T_\perext^{N_0}x\in\Kappa(\mathcal R_\perext)$, Corollary
\ref{cor-strip-3} gives that $N_0\leq k<N$ and
$T_\perext^kx\in\Kappa(\mathcal R_\perext)$ implies that
$(\bP_\perext)_0^k(x)$ is a hyperbolic strip. This strip cannot be
admissible as $T_\perext^kx\in\kappa_s(\mathcal R_\perext)$ and
$k<N$, hence such $k$ is some $m_{ij}$, proving (3).
\end{demo}

\begin{demof}{Proposition \ref{prop-ent}}
According to Proposition \ref{prop-adm-strip}, given $N$, $M$, $R$
and $S$, to determine an element of $\mathcal I(N,L,M,R,S)$ we need
to specify:
 \begin{enumerate}
  \item the integers $N_1,N_2$, $n_1,\dots,n_R$ and $m_{i1},m_{is(i)}$ for $i=1,\dots,R$;
  \item the itineraries $(P_\perext)^{N_1}(x)$, $(P_\perext)^{n_1-1}_{N_2}(x)$ (if $n_1>N_2$)  and
  $(P_\perext)^{N-1}_{n_r+1}(x)$;
  \item $(P_\perext)^{m_{i1}}_{n_i}(x)$,
  $(P_\perext)^{n_{i+1}}_{m_{is(i)}}(x)$ for each $i=1,\dots,r$;
  \item a choice among $4\cdot 2^R$.
 \end{enumerate}
Observe that $N_1\leq N_2\leq N_0$.
Using property (3) of Lemma \ref{lem-mij} (in particular $N_0$ is some
$m_{ij}$), it follows that:
 \begin{multline*}
   \#\left(\bigcup_{i=1}^r \rightb m_{i1}, m_{is(i)}\leftb
     \;\setminus\; \rightb N_1,N_2 \leftb \right)
     \geq \#\left(\bigcup_{i=1}^r \rightb m_{i1}, m_{is(i)}\leftb
     \;\setminus\; \rightb 0,N_0 \leftb \right) \\
     \geq \#\bigcup_{\tiny\begin{matrix} i=1,\dots,r\\1\leq j<s(i)\\m_{ij}\geq N_0\end{matrix}}
     \rightb
     m_{ij},m_{ij}+L\leftb\;\;
     \geq (S-R)(L-1)\geq (S-R)L-S
 \end{multline*}
recalling the definitions of $L$ and $S$. Hence, the number of
choices for those items is bounded by:
\begin{enumerate}
  \item $\binom{N}{2}\binom{N}{R}^3$ where $\binom{a}{b}=a!/b!(a-b)!$ is the binomial coefficient;
  \item[(2-3)] $C_*^{R+2}\exp \left((h_\top(T)+\eps_*)(N-(S-R)L-M+S))\right)$;
  \item[(4)] $4\cdot 2^R$.
 \end{enumerate}
Recalling that\footnote{$f(t)\sim g(t)$ iff
$\lim_{t\to\infty}f(t)/g(t)=1$ and $H(t)=-t\log
t-(1-t)\log(1-t)$.} $\binom{n}{\alpha n} \sim
\frac{1}{\sqrt{2\pi\alpha(1-\alpha)}} n^{-1/2} e^{H(\alpha)n}$ as
$n\to\infty$, i.e., $\log\binom{n}{\alpha n} \leq
H(\alpha)n+C(\alpha)$ and that $k\in\leftb
0,(n-1)/2\rightb\mapsto\binom{n}{k}$ is increasing, the stated
bound follows with:
 $$
    K_*(\rho,N) = 3H(\rho)+\rho\log 2+3N^{-1}\log N+N^{-1}\log
    4C(\rho).
 $$
\end{demof}

\subsection{Large Average Return Times and Entropy}
We are going to apply the previous estimates linking long return
times either to visits to the holes ($\mathcal
R\setminus\kappa(\mathcal R)$) or to low entropy. We will show
that for a suitable choice of the parameters of our constructions,
large entropy measures have finite average return time.

Recall the good return time $\tau:\kappa_s(R)\to\bar\NN$ (possibly
infinite) of Definition \ref{def-good-return-time}. We define
$\tau_n(x)$, $n\geq1$, inductively by $\tau_1(x)=\tau(x)$ and
$\tau_{n+1}(x)=\tau(T^{\tau_n(x)}(x))$ ($\tau_{n+1}(x)=\infty$ if
$\tau_n(x)=\infty$).

\newcommand\musup{\mu\!\!-\!\!\!\sup}

The essential supremum of a function $f$ over a subset
$X$ with respect to a measure $\mu$ is:
 $$
   \musup_{x\in X} f(x) := \inf_{X'= X [\mu]} \;\;\sup_{x\in X'} f(x)
 $$
where $X'$ ranges over the measurable subsets of $X$ such that
$\mu(X\setminus X')=0$ ($X$ and $f$ are assumed to be measurable). Our key estimate is the following:

\begin{proposition}\label{prop-tau}
There exist $h_2<h_\top(T)$ and $L_2<\infty$ with the following
property. Consider the Markov array $\mathcal R$ defined by
Proposition \ref{prop-sm-rec}. For any integer $L_\perext\geq L_2$, let
$(M_\perext,T_\perext,P_\perext,\mathcal R_\perext)$ be the
$L_\perext$-periodic extension of Definition \ref{def-perext}.
Then, for each $\mu\in\Proberg^{h_2}(T_\perext)$, the good return time
with respect to $P_\perext$ and $\mathcal R_\perext$, satisfies:
 $$
   \tau_*(\mu)= \musup_{x\in\mathcal R_\perext}\tau_*(x)<\infty \text{ where }
     \tau_*(x):=\limsup_{n\to\infty} \frac1n\tau_n(x).
 $$
\end{proposition}

\begin{remark}

(1) $\tau_*(x)<\infty$  a.e. already ensures that almost every point in
$\kappa(\mathcal R)$ has a good return. But we shall need more.

(2) The proof below does \emph{not} provide a semi-uniform bound on
$\tau_*$ as our estimates below depend on the speed of convergence
of some ergodic averages (see Remark \ref{rem-not-su} below).
\end{remark}

\begin{demof}{Proposition \ref{prop-tau}}
The first step of the proof fixes a Markov array $\mathcal R$ and
a periodic extension of $T$ and finds a candidate upper bound for
$\tau_*(\mu)$. The second step defines a language (a collection of
words of increasing lengths) with small entropy. The final step
shows that large average return times imply that this language is
enough to describe the measure. A large average can therefore
happen only for low entropy measures.

\step{1}{The Markov Array}

We apply Proposition \ref{prop-sm-rec} and obtain first numbers $\mu_0>0$
and $h_0<h_\top(T)$. We let $0<\eps_0<\min(h_\top(T),1)/200$ be small enough so that
in the notations of Proposition \ref{prop-ent}:
 \begin{equation}\label{eq-Kstar}
       K_*(\eps_0\mu_0) < \frac{\mu_0}{100} h_\top(T).
 \end{equation}
We pick $L_*$ to be so large that, for all $\ell\geq L_*$
 \begin{equation}\label{eq-L1-large}
    K_*(\eps_0\mu_0,\ell)<\frac{\mu_0}{100}h_\top(T).
 \end{equation}  
Proposition \ref{prop-sm-rec} now gives an integer
$L_2(T,\eps_0,L_*)\geq L_*$ and a Markov array $\mathcal R$ 
such that the following holds. 

For each $L_\perext\geq L_2$,
any $\mu\in\Proberg^{h_0}(T)$ can be lifted to an ergodic invariant probability
measure $\mu_\perext$ on the periodic 
extension $(M_\perext,T_\perext,P_\perext,\mathcal
R_\perext)$ satisfying:
 $$
   L_\perext \mu_\perext(\Kappa(\mathcal R_\perext))>\mu_0 \text{ and }
   L_\perext \mu_\perext(\mathcal R_\perext\setminus\kappa(\mathcal R_\perext))< \eps_0\mu_0.
 $$
Recall that $P^\mathcal R$
denotes the convex partition of $M$ generated by $P$ and $\mathcal R$.

We fix $\eps_*:=(\mu_0/100)h_\top(T)$ and define $C_*=C_*(P^\mathcal R,\eps_*)<\infty$ as in
Proposition \ref{prop-ent}. Note that $C_*$ does not depend on $L_\perext$. 
Hence, possibly after increasing $L_\perext$, we may assume that
 $$
   L_\perext>\frac{\log C_*} {(\mu_0/100)h_\top(T)}.
 $$
Fix $\ell_*\geq L_*$ so large that $3\log C_*/\ell_*<(\mu_0/100) h_\top(T)$.

\medbreak We omit the sharp subscript in the sequel so that
$M,T,P,\mu,\mu_0$ will denote in fact
$M_\perext,T_\perext,P_\perext,\mu_\perext,\mu_{0\perext}$
(in particular $L_\perext\mu_0$ is the original $\mu_0$). To
refer to the original $\mu$ or $\mu_0$, we shall write $\mu_\nos$
or $\mu_{0\nos}$. It will be a convenient exception to continue to
write $P^\mathcal R$ for $P_\nos^{\mathcal R_\nos}$.

 \medbreak
We let
 $$
   h_2:=\max(h_0,h_\top(T)\left(1-0.9L_\perext\mu_0\right))<h_\top(T)
 $$ 
and fix some $\mu\in\Proberg^{h_2}(T)$ together with $\mu_\perext$
as above. 
According to the Birkhoff ergodic theorem,  one can find $K_1\subset M$ and $L_1<\infty$ 
such that
 \begin{equation}\label{eq-mes-K1}
   \mu(M\setminus
K_1)<\eps_0\mu_0^2/(10^6\log\#P^{\mathcal R})
 \end{equation}
and, for all $x\in K_1$:
 \begin{equation}\label{eq-K1}
 \begin{gathered}
  \forall n\geq L_0:=\frac{\eps_0\mu_0}{1000} L_1\;\;\;
     \left|\frac1n \#\{0\leq k<n:T^kx\in\Kappa(\mathcal R)\} -
     \mu(\Kappa(\mathcal R)) \right| <
        \left(\frac{\mu_0}{1000}\right)^2. \\
  \quad\forall n\geq L_1\qquad\;
  \frac1n \#\{ 0\leq k<n:T^kx\in\mathcal R\setminus\kappa(\mathcal R)\}<
  \eps_0
  \mu_0.
 \end{gathered}
 \end{equation}

\begin{remark}\label{rem-not-su}
The above $L_1$ is the only estimate in the proof of this
proposition which does not seem semi-uniform.
\end{remark}

Increasing $\ell_*$ if necessary, we assume 
  $$
     \ell_*\geq (1000/\mu_0)L_1
  $$
and $\ell\binom{\ell}{\ell/\ell_*}
\leq e^{(L_\perext\mu_0/100)h_\top(T)\cdot \ell}$ for all $\ell\geq \ell_*$.

We set:
 \begin{equation}\label{eq-tauM}
 \tau_{\max}:=\frac{1000\log\#P^\mathcal R}{\mu_0} \ell_*.
 \end{equation}

To prove that $\tau_*(\mu)\leq\tau_{\max}$, we assume by
contradiction that
  \begin{equation}\label{eq-Mstar}
  M_1:=\{x\in\kappa(\mathcal R): \limsup_{n\to\infty}
  \frac1n\tau_n(x)>\tau_{\max}\} \text{ has positive $\mu$-measure}.
  \end{equation}

\step{2}{Low entropy language}

For each integer $\ell\geq1$ we define a set $C(\ell)$ of
$P^{\mathcal R}$-words of length $\ell$ as
 $$
    C(\ell):=\bigcup_{\small\begin{matrix} \ell_1+\dots+\ell_k=\ell\\
     k\leq \ell/\tau_{\max}\end{matrix}}
       C(\ell_1,\dots,\ell_k)
 $$
Here $C(\ell_1,\dots,\ell_k)$ is the set of all concatenations
$\gamma_1\dots\gamma_k$ where each $\gamma_m$ ($1\leq m\leq k$) is
a word (i.e., a finite sequence) of length $|\gamma_m|=\ell_m$
satisfying:
\begin{itemize}
 \item \emph{type 1 requirement:} $\gamma_m$ is an itinerary from $\mathcal I(\ell_m,L_\perext,M,R,S)$ (in the
 notation of Proposition \ref{prop-ent}) with:
  \begin{equation}\label{eq-TSRM0}
    \ell_m\geq \ell_*,\quad
    L_\perext(S-R)+M-S\geq \frac{98}{100}L_\perext\mu_0\ell_m
    \text{ and } R  < \frac{\min(h_\top(T),1)}{200}L_\perext\mu_0 \ell_m;
  \end{equation}
Recall that $L_\perext\mu_0$ is the original $\mu_0$, independent on $L_\perext$.
 \item \emph{type 2 requirement:} the sum of the lengths of these
 segments is less than $\frac{L_\perext\mu_0}{500\log\#P}\ell$.
\end{itemize}
Observe that the union defining $C(\ell)$ has at most $\ell \binom{\ell}{[\ell/\tau_{\max}]}
\leq e^{(L_\perext\mu_0/100)h_\top(T)\ell}$ terms.
It remains to bound $\#C(\ell_1,\dots,\ell_n)$.

By Proposition \ref{prop-ent}, the logarithm of $\#\mathcal
I(\ell_m,L_\perext,M,R,S)$ under condition (\ref{eq-TSRM0}) is bounded by
 \begin{multline}
    \left(h_\top(T)+\eps_*\right) \left(\ell_k-(L_\perext(S-R)+M-S)\right)
           + K(\eps_0\mu_0,\ell_k)\ell_k+(2R+3)\log C_*\\
     \leq h_\top(T)\left(1+\frac{L_\perext\mu_0}{100}\right)
       \left(1-\frac{98}{100}L_\perext\mu_0\right)\ell_k
       + \frac{2L_\perext\mu_0}{100}h_\top(T)\ell_k+3\log C_* \\
     \leq h_\top(T)\left(1 -
     \frac{94}{100}L_\perext\mu_0\right)\ell_k +3\log C_*\\
     \leq h_\top(T)\left(1-0.93L_\perext\mu_0\right)\ell_k,
 \end{multline}
Hence, 
 \begin{multline}\label{eq-C-ell}
   \# C(\ell) \leq \exp\left((L_\perext\mu_0/100)h_\top(T)\ell\right)
     \times \exp \left( h_\top(T)\left(1-0.93L_\perext\mu_0\right)\ell\right)\\
       \times (\#\mathcal P^{\mathcal R})^{( L_\perext\mu_0/
                500\log \#\mathcal P^{\mathcal R})\ell}
    \leq e^{h_2\ell}.
 \end{multline}

\step{3}{Consequence of Large Return Times}

We are going to show that, for all $x\in\kappa(\mathcal R)$, all large
enough integers $n$:
 \begin{equation}\label{claim-ret}
   \tau_n(x)>\tau_{\max} \cdot n \implies P^{\tau_n(x)}(x)\in
   C(\tau_n(x)).
 \end{equation}
Observe that this will imply that, any ergodic and invariant measure $\mu$
such that $\tau_*(\mu)>\tau_{\max}$ has entropy at most $h_2$
using Proposition \ref{prop-entropy-bound} with  (\ref{eq-C-ell})
and:
\begin{itemize}
  \item $M_0:=\{x\in M:\{n\geq0:T^{-n}x\in M_1\}$ is infinite$\}$
  (recall eq. (\ref{eq-Mstar}));
  \item $a_i(x):=\min\{j\geq i:T^{-j}x\in M_1\}$ for all $i\geq1$;
  \item $b_i(x):=\tau_{n}(T^{-a_i(x)}x)$ with $n$ a positive integer such that
  $\tau_{n_i}(x)\geq\max(a_i(x),\tau_{\max}\cdot n_i)$,
\end{itemize}
concluding the proof of Proposition \ref{prop-tau}.
We now prove (\ref{claim-ret}).

\medbreak

Let $x\in \kappa(\mathcal R)$. We consider a large integer $n$ such that 
$\tau_n(x)>\tau_{\max} \cdot n$. (\ref{eq-mes-K1}) 
and the ergodic theorem give:
 $$
 \frac1{\tau_n(x)}\#\{0\leq k<\tau_n(x):T^kx\notin
K_1\}<\eps_0\mu_0^2/(10^6\log\#P^\mathcal R).
 $$
Let $N:=\tau_n(x)$ and, for $k=0,\dots,n-1$, let $I_k$ be the
integer interval $\leftb \tau_k(x),\tau_{k+1}(x)\leftb$ and
$\ell_k:=\# I_k$.

Let $B_1\subset\leftb 0,n\leftb$ be the set of those integers
$0\leq k<n$ such that
 $$
 \#\{m\in I_k:T^mx\notin K_1\}\geq\frac{\eps_0\mu_0}{1000}\ell_k.
 $$
The union of those segments $I_k$ occupies only a small proportion
of $\leftb 0,N\leftb$:
 $$
    \sum_{k\in B_1}  \ell_k \leq
    \frac{1000}{\eps_0\mu_0} \times \frac{\eps_0\mu_0^2}{10^6\log\#P^{\mathcal R}} N
     \leq \frac{\mu_0}{1000\log\#P^{\mathcal R}} \cdot N.
 $$

Let $B_2\subset\leftb 0,n \leftb$ be the set of $k$'s such that
$\ell_k\leq \ell_*$. They also occupy a small
proportion:
 $$
    \sum_{k\in B_2} \ell_k \leq \ell_* n \leq
      \ell_* \frac{N}{\tau_{\max}} \leq \frac{\mu_0}{1000\log\#P^{\mathcal R}} \cdot N,
 $$
by the choice of $\tau_{\max}$.

Therefore the segments $I_k$ for $k\in B_1\cup B_2$ satisfy the
type 2 requirement in the definition of $C(\ell)$. It is enough to
prove that the remaining $I_k$s satisfy the type 1 requirement.

For such segments $I_k$, $p_1:=\min\{p\geq 0:T^{p+\tau_k(x)}x\in
K_1\}$ satisfies: 
 \begin{equation}\label{eq:p1}
    p_1\leq \frac{\eps_0\mu_0}{1000}\ell_k
 \end{equation}
by the definition of $B_1$. By the definition of $B_2$, 
 \begin{equation}\label{eq:lk-L0-L1}
   \ell_k\geq \ell_*\geq \frac{1000}{\mu_0}L_1=\frac{10^6}{\eps_0\mu_0^2} L_0.
 \end{equation}
This fulfills the first requirement of (\ref{eq-TSRM0}).

Hence, in the notations of Proposition \ref{prop-ent}:
 \begin{equation}\label{eq-N0}
    N_1:=\min\{j\geq0: T^{\tau_k(x)+j}x\in\Kappa(\mathcal R)\}
         \leq p_1+L_0 \leq \frac{\eps_0\mu_0}{500} \ell_k.
 \end{equation}
Also, by eq. (\ref{eq-K1}) and $\ell_k-p_1\geq L_1$:
 $$
    \#\{j\in \leftb\tau_k(x)+p_1,\tau_{k+1}(x)\leftb:
            T^jx\in R\setminus\kappa(\mathcal R)\}< {\eps_0\mu_0}
            (\ell_k-p_1).
 $$
Hence, using point (2) of Lemma \ref{lem-mij}:
 \begin{equation}\label{eq-R-bound}
   R \leq r' :=  \#\{j\in I_k:
            T^jx\in \mathcal R\setminus\kappa(\mathcal R)\}< 2\eps_0\mu_0 \ell_k.
 \end{equation}
Note that this implies $R\leq \ell_k\cdot \mu_0\min(h_\top(T),1)/200$, a part of the
type 1 requirement. It remains to show the lower-bond on 
$L_\perext(S-R)+M$.

First, similarly to (\ref{eq-R-bound}):
 $$
    \left|\#\{j\in I_k: T^jx\in\Kappa(\mathcal R)\}
           -\mu(\Kappa(\mathcal R))\ell_k\right| < \frac{\mu_0}{500}
           \ell_k.
 $$
Setting, again as in Proposition \ref{prop-ent}:
 $$
 N_0:=\min\{j\geq0:T^{\tau_k(x)+j}(x)\in\Kappa(\mathcal R)\text{ and }
    P^{j}(T^{\tau_k(x)}x)\text{ is hyperbolic}\}
 $$
(observe that $N_0$ might be large) and $S:=\#\{m_{ij}>N_0:i,j\}$
we get, using point (3) of Lemma \ref{lem-mij},
 $$
   S\geq s':= \#\{j\in \rightb\tau_k(x)+N_0,\tau_{k+1}(x)\leftb:
       T^jx \in \Kappa(\mathcal R)\}.
 $$
Also, $M:=N_2-N_1=[N_0]_\kappa-N_1$ where:
 $$
    [t]_\kappa:=\max\{n\in\leftb 0,t\leftb: T^nx\in\Kappa(R)\}
 $$
(we define $[t]_\kappa:=t$ if there is no such integer). 
To complete our estimate, we consider two cases.

\medbreak\noindent{\bf --- First case:} $N_0< p_1+L_1$. We use the
trivial bound $M\geq0$, (\ref{eq:p1}), (\ref{eq:lk-L0-L1}) and (\ref{eq-K1}) to get:
 $$
   s'\geq \#\{j\in\leftb \tau_k(x)+p_1+L_1,\tau_{k+1}(x)\leftb :
     T^jx\in\Kappa(R)\}
     \geq \frac{499}{500}\mu_0(\ell_k-p_1-L_1)
     \geq \frac{498}{500}\mu_0\ell_k
 $$
Hence,
  \begin{equation}\label{eq-TSM1}
   L_\perext S+M \geq \frac{498}{500}L_\perext\mu_{0}\ell_k
    \geq \frac{99}{100}L_\perext\mu_{0}\ell_k.
 \end{equation}

\medbreak\noindent{\bf --- Second case:} $N_0\geq p_1+L_1$. Using the
definition of $p_1$, $K_1$ and $L_0$:
 $$
  s'\geq \mu(\Kappa(R))\left((1-10^{-3})(\ell_k-p_1)-
  (1+10^{-3})(N_0-p_1)\right).
 $$
>From (\ref{eq:p1}),
 $$
   s' \geq \frac{998}{1000}\mu(\Kappa(\mathcal
   R))\ell_k-\frac{1001}{1000}\mu(\Kappa(\mathcal R))N_0.
 $$
Hence, using $(1001/1000)L_\perext\mu(\Kappa(\mathcal R))\leq1$
and $M=[N_0]_\kappa-N_1$:
 \begin{multline*}
 L_\perext S+M \geq \frac{998}{1000}L_\perext\mu(\Kappa(\mathcal
   R))\ell_k-\frac{1001}{1000}L_\perext\mu(\Kappa(\mathcal
   R))N_0+[N_0]_k-N_1 \\
   \geq \frac{998}{1000} L_\perext\mu_0\ell_k-(N_0-[N_0]_k)-N_1.
 \end{multline*}
In light of (\ref{eq-N0}), to prove that eq. (\ref{eq-TSM1}) also
holds in this second case it is enough to show:

\begin{claim}\label{claim-t}
For any $0\leq t\leq\ell_k$,
 $
     t-[t]_\kappa\leq \frac{\mu_0}{250}\ell_k.
 $
\end{claim}

\begin{demof}{claim \ref{claim-t}}
We distinguish two cases. First assume that $[t]_\kappa<p_1+L_0$.
Then $t$, the first visit to $\Kappa(R)$ after $[t]_\kappa$, is
bounded by the first visit after $p_1+L_0$, i.e., using (\ref{eq-N0}) and
$\ell_k\geq(1000/\mu_0^2)$ which follows from (\ref{eq:lk-L0-L1}):
 $$
    t-[t]_\kappa\leq t\leq p_1+L_0+\frac{1000}{999}\mu_0^{-1}\leq \frac{4}{1000}\mu_0 \ell_k
 $$
proving the claim in this case. Second we assume that
$[t]_\kappa>p_1+L_0$. Then:
 \begin{multline}
    (1-\mu_0/1000)\mu(\Kappa(\mathcal R))(t-p_1)\\
      \leq \#\{j\in\leftb\tau_k(x)+p_1,\tau_k(x)+t\leftb:T^jx\in\Kappa(\mathcal R)\} \\=
     \#\{j\in\leftb\tau_k(x)+p_1,\tau_k(x)+[t]_\kappa\rightb:T^jx\in\Kappa(\mathcal R)\}
     \\
     \leq (1+\mu_0/1000)\mu(\Kappa(\mathcal R))([t]_\kappa+1-p_1)
 \end{multline}
So $t-p_1\leq (1+3\mu_0/1000)([t]_\kappa+1-p_1)$. Hence,
 $$
   t -[t]_\kappa \leq \frac{3}{1000}\mu_0 [t]_\kappa +2\leq
   \frac{4}{1000} \mu_0\ell_k,
 $$ proving the claim.
\end{demof}

In both cases, eq. (\ref{eq-TSM1}) together with (\ref{eq-R-bound}) implies:
 \begin{equation}\label{eq-TSRM}
    L_\perext(S-R)+M-S  \geq \frac{98}{100} L_\perext\mu_0\ell_k.
 \end{equation}
This establishes the remaining part of the type 1 requirement on
$\bP_{\tau_k(x)}^{\tau_{k+1}(x)}(x)$ for all
$k\in\leftb0,n\leftb\setminus(B_1\cup B_k)$. Hence
$P^{\tau_n(x)}(x)$ belongs to $C(\tau_n(x))$, concluding the
proof of eq. (\ref{claim-ret}) and of Proposition \ref{prop-tau}.
\end{demof}

It also follows from the above proof that:

\begin{corollary}\label{coro-bdd-gap}
$\tau:\kappa_s(\mathcal R)\to\bar\NN^*$ has eventually bounded gaps in the sense of Appendix
\ref{appendix-jump} \wrt any large entropy measure.
\end{corollary}

\begin{demo}
Given $\mu$, a large entropy measure, we fix $\tau_{\max}$ as in
(\ref{eq-tauM}) and we proceed by contradiction assuming that for
each large $t>\tau_{\max}$, there is
a set of positive $\mu$-measure $S$ with the following property.
For each $x\in S$,  there exist sequences of integers $n_k\in\ZZ$ and
$m_k\in\NN^*$ such that:
 $$
    \tau_{m_k}(T^n_kx)> t\cdot m_k \text{ and }
    \sup_k \inf\{|i|:i\in \leftb n_k,n_k+\tau_{m_k}(T^n_kx)\}<\infty.
 $$
(the case of improper orbits is similar and easier and left to the reader). It
is now enough to apply Proposition \ref{prop-entropy-bound} using eqs. (\ref{eq-C-ell}-\ref{claim-ret})
to get that $h(f,\mu)\leq h_2$, a contradiction.
\end{demo}

\section{Proof of the Main
Results}\label{sec-lifting}

We finally prove the main results by building a Markov system from
the arbitrary concatenations of admissible strips and relating it
to the dynamics of the piecewise affine homeomorphism. This is done
in the analysis of large entropy measures and then used for the
other claims.

\subsection{Maximal Entropy Measures}\label{sec-main-thm}
We prove Theorem \ref{thm-max-ent} about the finite number of
maximum measures. 

\step{1}{Tower}

Fix a Markov array $\mathcal R$ as in Proposition
\ref{prop-sm-rec}, defining a $(\mathcal R,L_\perext)$-periodic
extension $(M_\perext,T_\perext,P_\perext,\mathcal R_\perext)$ of
$(M,T,P,\mathcal R)$. This may only increase the number of maximum
measures as it is a finite topological extension.
Let $\Sigma:=\Sigma(T_\perext,P_\perext)$ be its symbolic
dynamics (see Definition \ref{def:ds}). Corollary \ref{coro-conj-ds} 
shows that it is enough to prove the results for $\Sigma$.

We now build an invertible tower (see Appendix \ref{appendix-jump})
$\hS$ over $\Sigma$. This is done by defining a return time
$\tau:\Sigma_\tau\to\mathbb N^*$ for some
$\Sigma_\tau\subset\Sigma$.

\begin{definition}\label{def-admi-word}
An \new{extended admissible $P_\perext$-word}  is a word
$w_0\dots w_n$ over $P_\perext$ such that $[w_0\dots w_n]$ is an
admissible strip. $w_0\dots w_{n-1}$ is the associated
\new{admissible $P_\perext$-word}.
\end{definition}

For a sequence $A\in \Sigma$, we define inductively
 $$
   t_1(A) := \sup\{n\geq 1:A_0\dots A_n\text{ is an extended
   admissible word}\}\in\NN^*\cup\{-\infty,\infty\}
 $$
and $t_{n+1}(A)=t_1(\sigma^{t_n(A)}(A))$ (or $t_n(A)$
if it was infinite). Let
 $$
   \Sigma_\tau:=\{A\in\Sigma:
        \forall n\geq1\; t_n(A)\in\NN^*\}.
 $$

\medbreak

We tacitly exclude entropy-negligible subsets of points of $M_\perext$
and of $\Sigma$ (these correspond by Lemma \ref{coro-conj-ds}).

\begin{claim}\label{claim-Sigmak}
$\Sigma_\tau$ coincides with the set of $P_\perext$-itineraries of
the points $x\in\kappa_s(\mathcal R_\perext)$.
\end{claim}

\begin{demof}{the Claim}
$x$ has a finite good return time $m:=\tau(x)$ by Proposition
\ref{prop-tau}. Let $A\in\Sigma$ be its itinerary. 
Observe that $S^0:=A_0\dots A_{m-1}$ is admissible and
$T_\perext^{m}x\in\kappa_s(\mathcal R_\perext)$. By
induction, $A$ splits into a concatenation of admissible words
$S^0S^1S^2\dots$. The cylinders defined by the finite concatenations $S^0S^1\dots
S^kA_{\tau_{k+1}(x)}$, $k\geq0$, are hyperbolic strips.

Clearly $t_1(A)\geq\tau(x)$. Now, if $H:=(\bP_{\perext})_0^n(x)$ is hyperbolic with $n>m$, then
$H$ is contained in $S^0$ (an $m$-admissible strip) and $T^mH$
meets the hyperbolic strip $S^1$. Hence $H$ cannot be admissible,
proving that $t_1(A)=\tau(x)<\infty$. $t_n(A)<\infty$ for all
$n\geq1$ follows from invariance. Hence, $A\in\Sigma_\tau$.

For the converse, let $A\in\Sigma_\tau$ and denote by $x$ the point with
itinerary $A$. $[A_0\dots A_m]$ is an admissible strip for
$m=t_1(A)$ so $A_0\in\mathcal R$. If $k=t_n(A)$,
then $[A_0\dots A_k]$ is a concatenation of admissible strips,
hence a $k$-strip. $k$ being arbitrarily large, it follows that
$W^s(x)$, the intersection of the previous strips,
must cross $A_0$, proving $x\in\kappa_s(\mathcal R)$.
\end{demof}

By construction, for all $A\in\Sigma_\tau$, $\sigma^{t_1(A)}(A)\in\Sigma_\tau$. Hence $t_1$
is a return time and defines an invertible tower $\hat T:\hS\to\hS$ in the
sense of Appendix \ref{appendix-jump}. Moreover, Corollary
\ref{coro-bdd-gap} shows that any large entropy measure $\mu$ on
$\Sigma(T_\perext,P_\perext)$ has eventually bounded gaps in the
sense of Definition \ref{def-bdd-gap}. By Proposition
\ref{prop-lift-abstract}, any such measure can be lifted to $\hS$
and any invariant probability measure of $\hS$ is a finite extension of
one in $\Sigma(T_\perext,P_\perext)$ (in particular both measures have the same
entropy).

It follows also that $h(\hat T):=\sup_\mu h(\sigma|hS,\mu)=
h_\top(\Sigma)$ so that maximum measures of $\Sigma$ lift to
maximum measures of $\hat T$.

To prove the theorem it is therefore enough to show that the tower $\hS$ has
finitely many ergodic measures of maximal entropy.

\step{2}{Markov Structure}

Recall the following definition
(see \cite{GS} for background):

\begin{definition}
A \new{Markov shift} is a space of sequences
 $$
   \Sigma(\mathcal G):=\{ x\in V^\ZZ:\forall n\in\ZZ\;
      x_n\to x_{n+1} \text{ in }\mathcal G\}
 $$
where $\mathcal G$ is an oriented graph with a countable set of
vertices $V$ together with the left shift $\sigma$. 
\end{definition}

We also recall that a graph $\mathcal G$ as above is \new{(strongly) irreducible} if for
every $(A,B)\in V^2$, there is a path from $A$ to $B$ on $\mathcal
G$. An \new{(strongly) irreducible component} is a subgraph with
this property maximum \wrt inclusion.

It will be convenient to say that $w'$ is a \emph{follower} of $w$ if
the two are admissible words $w,w'$ in the sense of Definition \ref{def-admi-word},
and $s$ being the first symbol of $w'$,  the concatenation $ws$ is an
extended admissible word.

\begin{claim}
Let $\mathcal G$ be the oriented graph with vertices $(w,i)$ where $w$ is
any admissible word and $0\leq i<|w|$ ($|\cdot|$ is the length of
the word) and arrows:
 $$
    (w,i)\to(w,i+1) \text{ if }i+1<|w|,\;
    (w,|w|-1)\to(w',0) \text{ if $w'$ is a follower of $w$}.
 $$
The tower $\hS$ is measurably conjugate\footnote{That is, there is a bimeasurable bijection $\psi:\hS\to\Sigma(\mathcal G)$ such that $\psi\circ\hat T=\sigma\circ\psi$.} to the Markov shift $\Sigma(G)$. 
\end{claim}

\begin{demof}{the Claim}
Define $p:\Sigma(\mathcal G)\to\hS$ by $p(\alpha)=(A,\omega)\in\Sigma\times\{0,1\}^\ZZ$ where, if $\alpha_n=:(w_0\dots
w_{\ell-1},k)$, then $A_n=w_k$ and $\omega_n=1$ if and only
if $k=0$. Observe that the sequence $A$ thus obtained is a
concatenation of admissible words so $A\in\Sigma$. Also, whenever
$m$ and $n$ are two successive integers with
$\omega_m=\omega_n=1$, $A_nA_{n+1}\dots A_m$ is an admissible
word. Finally $\omega_m=1$ for infinitely many positive and
negative integers $m$. The proof of Claim \ref{claim-Sigmak} shows that
$\sigma^mA\in\Sigma_\tau$ for all such
$m$, so that $(A,\omega)\in\hS$. Thus $p$ is well-defined and  is
clearly a measurabke conjugacy.
\end{demof}

\step{3}{Conclusion}

\begin{demof}{Theorem \ref{thm-max-ent}}
$\mathcal G$ has at most one
irreducible component for each 
vertex of the type $(w_0,0)$. These corrrespond to the finitely many rectangles in
the Markov array $\mathcal R$. Hence, by a classical result of
Gurevi\v{c} \cite{Gurevic}, $\Sigma(\mathcal G)$ with
$h(\Sigma(\mathcal G))<\infty$ has finitely many maximum measures,
proving the Main Theorem.
\end{demof}

\subsection{Number of Periodic Points}\label{sec-nbr-per}

We prove Proposition \ref{prop-periodic} about the number of
periodic points. We use the construction of the proof of Theorem
\ref{thm-max-ent}. To prove the lower bound
(\ref{eq:mult-bound}). It is enough to prove it for
$T_\perext$,  as $T_\perext$ is a finite extension of $T$. We assume by
contradiction that the number of points fixed under $T_\perext^n$
is such that for any integer $p\geq1$, there is a sequence
$n_k\to\infty$ of multiples of $p$ such that
 \begin{equation}\label{eq-small-NT}
     \lim_{n_k\to\infty} \frac{N_{T_\perext}(n_k)}{e^{n_k h_\top(T)}} = 0
 \end{equation}
In the following we denote
$(M_\perext,T_\perext,P_\perext,\mathcal R_\perext)$ by
$(M,T,P,\mathcal R)$.

The starting point is the following estimate for $\Sigma(\mathcal G)$. 
Consider a maximum measure. It is carried on one irreducible component.
For simplicity, we replace $\mathcal G$ by that irreducible component. By
Gurevi\v{c} \cite{Gurevic}, the existence of a maximum measure for
$\Sigma(\mathcal G)$ implies that $\mathcal G$ is \emph{positive
recurrent} with parameter $R=e^{-h_\top(T)}$. By Vere-Jones \cite{VJ},
this implies that the number $N_{\mathcal G}(n)$ of  loops of
length $n$ based at a given vertex satisfies, for some positive
integer $p$:
 \begin{equation}\label{eq-VJ}
   \lim_{n\to\infty,\; p|n} N_{\mathcal G}(n)e^{-h_\top(T)n} = p
 \end{equation}
Each $n$-periodic sequence $A$ in the symbolic dynamics $\Sigma$
is associated to a closed, convex set
 $$
   \bigcap_{j\geq0} \overline{T_\perext^j[A_{-j}\dots
A_j]}
 $$
invariant under $T_\perext^{n}$.
This set contains at least one point fixed by $T_\perext^n$
which we denote by $\pi(A)$. It remains to show that $\pi$ does
not identify too much points.

\medbreak

Consider $\pi$ from the set $\Sigma(n)$ of $n$-periodic sequences
to the set $M_\perext(n)$ of $n$-periodic $T_\perext$-orbits. Our
assumption (\ref{eq-small-NT}) implies that, for some sequence
$m_n\to\infty$ (where $p|n$) this map is at least $m_n$-to-$1$ on
a subset $\Sigma'(n)$ of $\Sigma(n)$ with cardinality at least
$e^{n h_\top(T)}/3$. We use the following observation:

\begin{lemma}\label{lem-mult2}
Let $A^1,\dots,A^m\in\Sigma$ be such that
$\pi(A^1)=\dots=\pi(A^m)=:x$. If the finite words $A^i_0\dots
A^i_{n-1}$, $i=1,\dots,m$, are pairwise distinct, then:
 \begin{enumerate}
   \item either $T^kx$ is a vertex of $P_\perext$ for some
   $0\leq k<n$;
   \item or there exist $r\geq (m-1)/2$ distinct integers $0\leq n_1<\dots<n_r<n$ such
   that for all $(k,l)$ with $1\leq k<\ell<n$, $T^{n_k}x$ lies on the interior of an edge of $P_\perext$. Moreover,
   if $v_k$ is the direction of the open edge containing $T^{n_k}x$, then
   the image by $T^{n_\ell-n_k}$ of $v_k$ is transverse to
   $v_\ell$.
 \end{enumerate}
\end{lemma}

\begin{demo}
Let $A^i_0\dots A^i_{n-1}$, $i=1,\dots,m$ be finite words as
in the above statement. We show that the failure of (1) implies
(2).

Each word $A^i_0\dots A^i_{n-1}$ defines an element of
$P_\perext^{n}$ containing $x$ in its closure so
$\mult(x,P_\perext^n)\geq m$. We assume that (1) fails. Observe that
 \begin{enumerate}
  \item[(i)] $\mult(x,P_\perext^{k+1})=\mult(x,P_\perext^k)$ if $T^kx$ is in the
  interior of an element of $P_\perext$ or if for all $A\in P_\perext^k$,
   $$
    P_\perext^kT_\perext^k(A\cap B(x,\eps))\subset B \text{ for some }B\in P_\perext
    \text{ and }\eps>0.
   $$
   \item[(ii)] $\mult(x,P_\perext^{k+1})\leq \mult(x,P_\perext^k)+2$.
 \end{enumerate}
(ii) uses that $T_\perext$ is a piecewise affine surface
homeomorphism so the preimage of an edge may locally divide in at
most two subsets at most two of elements of $P_\perext^k$ touching
$x$.

This implies that $\mult(x,P_\perext^n)\leq 1+2\#\{0\leq k<n:T_\perext^kx$
is on an edge of $P_\perext\}$. The Lemma follows.
\end{demo}

\begin{demof}{Proposition \ref{prop-periodic}}
Only $\text{const}\cdot n$ points of $T_\perext$ can satisfy
assertion (1) in the above Lemma. As
$h_\mult(T_\perext,P_\perext)=0$, their preimages in $\Sigma'(n)$
are in subexponential number. The remainder of $\Sigma'(n)$
corresponds to points $x$ whose orbit stays off the vertices of
$P_\perext$ and which admit distinct sequences $A^i\in\Sigma'(n)$,
$i=1,\dots,m_n$ with $\pi(A^i)=x$.

Given such an $x$, fix $0<n_1<\dots<n_r<n$ as in point (2) of the Lemma. Pick $j$
such that $0\leq n_{j+1}-n_j\leq 2n/(m_n-1)$. Thus $T_\perext^{n_j}x$ is
a vertex of $P^{n_{j+1}-n_j+1}$. The number of such vertices, for
given $n_j$, is bounded by $\text{const}\cdot
\#P_\perext^{2n/(m_n-1)}$. Taking into account the choice of
$n_j$, the number of such $x$s is bounded by:
$e^{(2/m_n)(h_\top(T)+\eps)n}$ for all large $n$. Thus, for 
large multiples $n$ of $p$,
 $$
   \#\Sigma(n) \leq 3\#\Sigma'(n) \leq 3e^{(2/m_n)(h_\top(T)+\eps)n}
    + 3e^{\eps n}.
 $$
As $m_n\to\infty$, this contradicts the Vere-Jones
estimate (\ref{eq-VJ}), proving the lower bound
(\ref{eq:mult-bound}) and Proposition \ref{prop-periodic}.
\end{demof}

{

\subsection{Entropy away from the singularities}\label{sec-horseshoes}
Proposition \ref{prop:lsc-entropy} is a corollary of the proof of
Theorem \ref{thm-max-ent} using the following result:

\begin{proposition}[B.M. Gurevi\v{c} \cite{Gurevic}]
Let $G$ be a countable, oriented graph. Let $G_0\subset G_1\subset\dots$
be a non-decreasing sequence of finite subgraphs exhausting $G$: any vertex and
any arrow of $G$ belong to $G_n$ with $n$ large enough.

Then $\lim_{n\to\infty} h_\top(G_n)=h_\top(G)$.
\end{proposition}

\begin{proof}[Proof of Proposition \ref{prop:lsc-entropy}]
In the proof of Theorem \ref{thm-max-ent}, one has shown that
there is a countable oriented graph $\mathcal G$ such that the corresponding
countable state Markov shift whose maximum measures have entropy $h_\top(T)$.

Let $\mathcal G_n$ be the finite subgraph defined
by keeping only the vertices and arrows of $\mathcal G$ which are on a loop of length
at most $n$ and based at one of the finitely many symbols representing an
element of the  Markov array. The sequence $\mathcal G_n$ exhausts $\mathcal G$.
The above proposition therefore ensures that, for any $\eps>0$, there is some
$\mathcal G_N$ with topological entropy at least $h_\top(T)-\eps$.

The projection of this subshift is a compact invariant subset $K$ of $M$ which 
contains only points with infinitely many visits to the controlled set 
$\kappa(\mathcal R)$ in the future and in the past. If $K$ met the singularity
lines of $T$, there would be such a point $x$ with the additional property that 
$Tx\in\partial W^u(Tx)$ or $T^{-1}x\in W^s(T^{-1}x)$. But this would prevent
any future or past visit to $\kappa(\mathcal R)$, a contradiction.

Finally, the above construction makes the isomorphism with $\Sigma(\mathcal G_n)$
obvious (it is indeed one-to-one as $K$ does not meet the boundary of $P$), 
but the latter is a subshift of finite type as $\mathcal G_n$ is finite.
 \end{proof}

\appendix

\section{Bounds on Metric Entropy}\label{sec-ent-bound}

We recall some standard notations and facts about entropy and a few consequences.
$(X,\mathcal B,\mu)$ is a probability space.
$H_\mu(P):=-\sum_{A\in P}\mu(A)\log\mu(A)$ denotes the mean entropy of a partition
(we shall leave implicit all measurability assumptions).
For $Y\subset X$,  $(\mu|Y)(\cdot):=
(\mu(Y))^{-1}\mu(\cdot\cap Y)$ (zero if $\mu(Y)=0$). 
For a sub-$\sigma$-algebra $\mathcal A$ of
$\mathcal B$, the \emph{conditional entropy} is:
 $$
   H_\mu(P|\mathcal A) := \int_X -\sum_{A\in P} 1_A\log E(1_A|\mathcal A)\,
d\mu,
 $$
where $E(\cdot|\mathcal A)$ is the conditional expectation with respect to $\mu$.

First, if $P$ is a partition, $Y\subset X$ and $\mathcal A$ is
a sub-$\sigma$-algebra of $\mathcal A$, then,
 \begin{multline}\label{eq:H-joint}
   H_\mu(P\vee\{Y,X\setminus X\}|\mathcal A) \\ \leq H_\mu(\{Y,X\setminus X\}|\mathcal A)
     + \mu(Y) H_{\mu|Y}(P|\mathcal A) + \mu(X\setminus Y) H_{\mu|(X\setminus Y)}(P|\mathcal A).
 \end{multline}

Second, the entropy of a measure can be
computed as the average of the entropies given the past. More
precisely, we have the following statement:

\begin{lemma}\label{lem-HPP}
Let $\mu$ be an invariant probability measure for some
bimeasurable bijection $T:X\to X$. Let $P$ be a finite, measurable
partition. Then:
 \begin{equation}\label{eq-HPP}
    h(T,\mu,P) = \int_X -\sum_{A\in P} 1_A\log E(1_A|P^-) \, \mu(dx)
 \end{equation}
where  with
$P^-$ is the \new{past partition}
generated by $T^nP$, $n\geq1$.

In particular, if $N(n,x,P) = \#\{P_0^{n-1}(y):y\in X$ and
$P_{-\infty}^{-1}(y) = P_{-\infty}^{-1}(x)\}$ where
$P_a^b(x):=(A_n)_{a\leq n\leq b}$ with $T^nx\in A_n$,
then:
\begin{equation}\label{coro-ent-N}
   h(T,\mu,P) \leq \int_X \frac1n\log N(n,x,P)
   \, \mu(dx).
\end{equation}
\end{lemma}

\begin{demo} For eq. (\ref{eq-HPP}), see,
e.g., \cite[Ex. 4(b) p.243]{Petersen} for entropy as an average of
conditional {information}.

Observe that
 $$
    E\biggl(1_A\log E(1_A|P^-)\biggm| P^-\biggr)(x) =
       E(1_A|P^-)(x)\log E(1_A|P^-)(x).
 $$
Hence the integrand in (\ref{eq-HPP}) can be replaced by the
above right hand side. 
Eq. (\ref{coro-ent-N}) now follows from the standard bound: 
 $$
  -\sum_{A\in P^N}E(1_A|P^-)(x)\log E(1_A|P^-)(x)  \leq \log\# \{A\in P^N:A\cap
W^u(x)\ne\emptyset\}.
 $$
\end{demo}

In combination with Rudolph's backward Vitali Lemma \cite[Theorem
3.9 p.33]{Rudolph}, this yields the convenient estimate:

\begin{proposition}\label{prop-entropy-bound}
Let $\mu$ be an ergodic, $\sigma$-invariant probability measure on
$\mathcal A^\ZZ$ with finite alphabet $\mathcal A$. Assume that
there exist a measurable family of subsets $W(A^-,\ell)\subset
P^\ell$ (for $A^-\in\mathcal A^{\ZZ_-}$, $\ell\geq1$) with
cardinality bounded by $Ce^{H\ell}$ and a subset $\Sigma_0 \subset
\mathcal A^\ZZ$ of positive measure such that, for all
$A\in\Sigma_0$, there are sequences of integers $a_i=a_i(A),
b_i=b_i(A)$, $i\geq1$, (depending measurably on $A$) satisfying:
 \begin{enumerate}
  \item $\lim_{i\to\infty} b_i-a_i=\infty$ ;
  \item $\sup_i \inf\{|k|:k\in\leftb a_i,b_i\rightb\}<\infty$;
  \item $A_{a_i}A_{a_i+1}\dots A_{b_i-1}\in W(\dots A_{a_i-2}A_{a_i-1},b_i-a_i)$
 \end{enumerate}
 Then
 $$
    h(\sigma,\mu)\leq H.
 $$
\end{proposition}

Condition (2) above means that the intervals $\leftb a_i,b_i\rightb$
do not escape to infinity: they all intersect some $[-R,R]$ for some $R$ large. 

\begin{demo}
To apply  Rudolph's Backward Vitali Lemma, we need
 \begin{equation}\label{eq-BVL}
a_i(A)\leq 0\leq b_i(A)
 \end{equation}
for all large enough $i$, for all $A\in\Sigma_0$. By passing to
subsequences, depending on $A$, we can assume the existence of the
(possibly infinite) limits $\lim_{i\to\infty} a_i(A)$,
$\lim_{i\to\infty} b_i(A)$ for all $A\in\Sigma_0$. Assume for
instance that $\lim_i a_i(A)=-\infty$ and $\lim_{i\to\infty}
b_i(A)<0$ for a.e. $A\in\Sigma_0$, the other cases being similar
or trivial. By assumption, $\inf_i b_i(A)>-\infty$ for all
$A\in\Sigma_0$. Restricting $\Sigma_0$ we can assume that this
infimum is some fixed number $b\in\ZZ$. Replacing $\Sigma_0$ by
$\sigma^{\min(b,0)}\Sigma_0$ ensures eq. (\ref{eq-BVL}).

Rudolph Lemma implies that for any $\eps>0$, for $\mu$-a.e. $A$,
for all large enough integers $n$, one can find a disjoint cover
of a fraction at least $1-\eps$ of $\leftb 0,n\leftb$ by at most
$\eps n$ intervals $\leftb a_i,b_i\rightb$ such that:
 $
    A_{a_i}\dots A_{b_i} \in W(\dots A_{a_i-2}A_{a_i-1},b_i-a_i)
 $
Applying eq. (\ref{coro-ent-N}) with:
 $$
   N(A,n)\leq \binom{n}{2\eps n} e^{Hn} \times \#\mathcal A^{\eps
   n}.
 $$
gives that $h(\sigma,\mu)\leq H+3\eps \log\eps+\eps \log
\#\mathcal A$. We conclude by letting $\eps\to0^+$.
\end{demo}

\section{Tower Lifts}\label{appendix-jump}

We study towers from a point of view closely related to that of
Zweimuller \cite{Zweimuller}. Let $T$ be an ergodic invertible
transformation of a probability space $(X,\mu)$ and let $B$ be a
measurable subset of $X$. A \emph{return time} is a function
$\tau:B\to\bar\NN^*:=\{1,2,\dots,\infty\}$ which is measurable and
such that $T^{\tau(x)}(x)\in B$ for all $x\in B$ with
$\tau(x)<\infty$ (but $\tau$ is not necessarily the first return
time).

We are interested in lifting $T$-invariant measures to the
following \new{invertible tower}:
 \begin{multline}
   \hat X:=\{(x,\omega)\in X\times\{0,1\}^\ZZ:
   \omega_n=1\implies T^nx\in B \text{ and } \\
   \tau(T^nx)=\min\{k\geq1:\omega_{n+k}=1\}\,\}\setminus\hat X_*
 \end{multline}
with $\hat T(x,\omega)=(Tx,\sigma(x))$ and $\hat X_*$ is the set
of $(x,\omega)$ with only finitely many $1$s either in the future
or in the past of $\omega$. 

Observe that:
 \begin{equation}\label{eq:tower-uniq}
  (x,\omega),(x,\omega')\in \hat X \text{ and }
  \omega_n=\omega'_n=1 \text{ for some $n$ }
  \implies \forall k\geq n \quad \omega_k=\omega'_k.
 \end{equation}

$(\hat X,\hat T)$ is an extension of a
subset of $(X,T)$ through $\hat\pi:\hat X\to X$ defined by
$\hat\pi(x,\omega)=x$.

\begin{remark}
The jump transformation $T^\tau:\{x\in B:\tau(x)<\infty\}\to B$ is
defined by $T^\tau(x):=T^{\tau(x)}(x)$. It is closely related to
$\hat T$. Indeed, $T^\tau$ is isomorphic to the first return map of $\hat T$ on
$[1]:=\{(x,\omega)\in\hat X:\omega_0=1\}$ so any $\hat
T$-invariant probability measure gives by restriction and
normalization a $T^\tau$-invariant probability measure (see
\cite{Zweimuller}).
\end{remark}

Such lifting requires that $\tau$ be "not too large" (see
\cite{Zweimuller} where the classical integrability condition is
studied). Our condition is in terms of the following "iterates" of
$\tau$: the functions $\tau_m:B\to\bar \NN^*$, $m\geq1$, are
defined, as before, by:
 $
   \tau_1:=\tau \text{ and } \tau_{m+1}(x):=\tau(T^{\tau_m(x)}(x))
   \text{ if $\tau_m(x)<\infty$, $\tau_{m+1}(x):=\infty$ otherwise.}
 $

\begin{definition}\label{def-bdd-gap}
$x\in X$ has an \new{improper orbit} if
 \begin{equation}\label{eq-past-recu}
   n(x):=\{n\in\NN:T^{-n}x\in B \text{ and }
   \forall m\geq1 \tau_m(T^{-n}x)<\infty\} \text{ is finite.}
 \end{equation}
$x\in X$ has \new{$t$-gaps} for some $0<t<\infty$ if $x$ has an
improper orbit or if there exist two integer sequences $(n_k)_{k\in\NN}$
and $(m_k)_{k\in\NN}$, $m_k>0$ for all $k\geq0$, such that:
   $$
     \forall k\geq0\quad \tau_{m_k}(T^{n_k}x) \geq \max(t\cdot m_k, k) \text{ and }
     \sup_{k\geq1} \min\{|i|:i\in [n_k, n_k+\tau_{m_k}(T^{n_k}x)]\} < \infty.
   $$
A measure has \new{eventually bounded gaps}, if for some
$t<\infty$, the set of points in $X$ with $t$-gaps has zero
measure.
\end{definition}

Note that $\tau(T^nx)=\infty$, for a single $n$, implies that $x$
has $t$-gaps for any $t<\infty$.

\begin{proposition}\label{prop-lift-abstract}
Let $T:X\to X$ be a self-map with a return time $\tau:B\to\NN^*$. Then:
 \begin{itemize}
  \item every $T$-invariant ergodic probability measure $\mu$ with
  eventually bounded gaps
  can be lifted to a $\hat T$-invariant ergodic
  probability measure on $\hat X$;
  \item any $\hat T$-invariant, ergodic probability measure $\hat\mu$
  is a finite extension of the $T$-invariant
  measure $\hat\pi(\hat\mu)$.
 \end{itemize}
\end{proposition}

\begin{demof}{Proposition \ref{prop-lift-abstract}}
We  first prove the existence of a lift for  $\mu$ like above. We
follow the strategy of \cite{Zweimuller} and \cite{KellerLift}
(which was inspired by constructions of Hofbauer) and define the
following \emph{non-invertible tower} to get a convenient topology:
 \begin{align*}
   &\tilde X:=\{(x,k,\tau)\in X\times\NN\times\NN: \exists y\in B\;
      \tau(y)=\tau,\; k<\tau\text{ and } x=T^ky\} \\
   &\tilde T(x,k,\tau):=(T(x),k+1,\tau) \text{ if }k+1<\ell,\;
   (T(x),0,\tau(T(x))) \text{ otherwise}.
 \end{align*}
For any integer $K$, we write $\tilde X_K:=\{(x,k,\tau)\in\tilde X:k=K\}$, $\tilde
X_{\leq K}:=\bigcup_{k\leq K} \tilde X_k$ and define
$\tilde\pi(x,k,\tau)=x$. Observe that $\tilde\pi\circ\tilde
T=T\circ\tilde\pi$ and that $(\hat X,\hat T)$ is a natural
extension of $(\tilde X,\tilde T)$ through
$(x,\omega)\mapsto(x,k,\ell)$ with $k\geq 0$ minimal such that
$\omega_{-k}=1$ and $\ell=\tau(T^{-k}x)$. Hence it is enough to
lift $\mu$ to $\tilde X$.

Fix $t<\infty$ such that the set of  points of $X$ with $t$-gaps has zero
$\mu$-measure. Let $\tilde\mu_0$ be the  probability measure
defined by
 $$
   \tilde\mu_0(\{(x,0,\tau(x)):x\in A\})=\mu(A) \text{ for all Borel sets }A
 $$
(sets disjoint from the above ones have zero
$\tilde\mu_0$-measure). We have $\tilde\pi(\tilde\mu_0)=\mu$ but,
except in trivial cases, $\tilde T_*\tilde\mu_0\ne\tilde\mu_0$ so we
consider:
 $$
   \tilde\mu_n:=\frac1n\sum_{k=0}^{n-1} \tilde T^k\tilde\mu_0
 $$
and try to take some accumulation point $\tilde\mu$. We identify
$\tilde\mu_n$ with its density with respect to $\tilde\mu_\infty$,
the $\sigma$-finite measure defined by
 $$
 \tilde\mu_\infty(\{(x,k,\tau(T^{-k}x)):x\in A\text{ and }k<\tau(T^{-k}x)\})=\mu(A\cap\{\tau>k\})
 $$
for all Borel sets $A\subset B$ and all $k\geq0$. As $\tilde\pi(\tilde\mu_n)=\mu$, we
must have:
 $$
    \frac{d\tilde\mu_n}{d\tilde\mu_\infty} \leq 1.
 $$
Using the Banach-Alaoglu theorem, i.e., the weak star compactness
of the unit ball of $L^\infty(\tilde\mu_\infty)$ as the dual of
$L^1(\tilde\mu_\infty)$, we get an accumulation point of the $\tilde\mu_n$,
i.e., a measure $\tilde\mu$ on $\tilde X$ with 
$d\tilde\mu/d\tilde\mu_\infty\leq 1$ such that, for some
subsequence $n_k\to\infty$,
 \begin{equation}\label{eq-mutilde}
   \forall f\in L^1(\tilde\mu_\infty)\; \lim_{k\to\infty}
     \int f \, d\tilde\mu_{n_k} = \int f\, d\tilde\mu.
 \end{equation}

Observe that $\tilde\mu$ is $\tilde T$-invariant: indeed, eq.
(\ref{eq-mutilde}) together with the $T$-invariance of $\mu$
implies that $d\tilde\mu\circ\tilde T^{-1}/d\tilde\mu \leq 1$
whereas $\tilde\mu\circ\tilde T^{-1}(\tilde X)=\tilde\mu(\tilde
X)$ so the previous inequality must be an equality
$\tilde\mu$-almost everywhere.

This invariance and the ergodicity of $\mu$
implies that $\tilde\pi\tilde\mu=\alpha\mu$ for some
$0\leq\alpha\leq1$. It remains to prove that $\tilde\mu\ne 0$ so
that it can be renormalized into the announced lift of $\mu$.
Assume by contradiction that $\tilde\mu=0$. Hence, for any
$L<\infty$:
 $$
    \int 1_{\tilde X_{\leq L}} \, d\tilde\mu_{n_k}
     = \int \frac1{n_k}\#\{ 0\leq k<n_k: \tilde
     T^k(x,0,\tau(x))\in \tilde X_{\leq L} \} \, d\mu
     \to 0
 $$
So, possibly for a further subsequence:
 \begin{equation}\label{eq-vanish}
    \frac1{n_k}\#\{ 0\leq k<n_k: \tilde
     T^k(x,0,\tau(x))\in \tilde X_{\leq L} \} \to 0
     \text{ $\mu$-a.e.}
 \end{equation}
Now,
 $$
   \#\{ 0\leq k<n: \tilde
     T^k(x,0,\tau(x))\in \tilde X_{0} \} < \eps n
     \implies  \tau_{\eps n}(x) \geq n.
 $$
Hence eq. (\ref{eq-vanish}) implies that $x$ (in fact any of its
preimages in the natural extension) has $t$-gaps for all $t>0$,
contradicting the assumption on $\mu$.

\medbreak

We now show that any $\hat T$-invariant, ergodic probability
measure $\hat\mu$ is a finite extension of
$\mu:=\hat\pi(\hat\mu)$. By definition of $\hat X$,
$\hat\mu([1])>0$ where $[1]=\{(x,\omega)\in\hat X:\omega_0=1\}$.
Assume that there is some positive measure subset $S\subset\hat X$,
and some number $K$ of measurable functions:
 $$
  \omega^1,\dots,\omega^K:S\to\{0,1\}^\ZZ
 $$
such that, for all $x\in S$, $(x,\omega^i(x))\in\hat X$, $\omega^i(x)\ne\omega^j(x)$ for $i\ne j$
and, for all $j=1,\dots,K$:
 $$
   \lim_{n\to\infty} \frac1n\#\{0\leq
   k<n:\omega^j_{-k}(x)=1\}=\hat\mu([1]).
 $$
If $K\cdot \hat\mu([1])>1$, then, for a.e. $x\in S$, there exist two distinct indices 
$j,j'\in\{1,\dots,K\}$ and arbitrarily large integers $n_k\to\infty$ such that
$\omega^j_{-n_k}(x)=\omega^{j'}_{-n_k}(x)$. But this implies
$\omega^j(x)=\omega^{j'}(x)$ by (\ref{eq:tower-uniq}). The
contradiction proves $K\leq \hat\mu([1])^{-1}<\infty$: $\hat\mu$
is a finite extension of $\mu$.
\end{demof}

\section{Examples}\label{sec-examples}

\noindent{\bf Positive Multiplicity Entropy}

\begin{example}[{\bf see \cite{BuzziAffine}}]\label{ex-pw-d2-mult}
There exists a continuous, piecewise affine surface map $(M,T,P)$
with $h_\mult(T,P)>0$ and $h_\top(T)=0$.
\end{example}

\begin{figure}
\centering
\includegraphics[width=5cm]{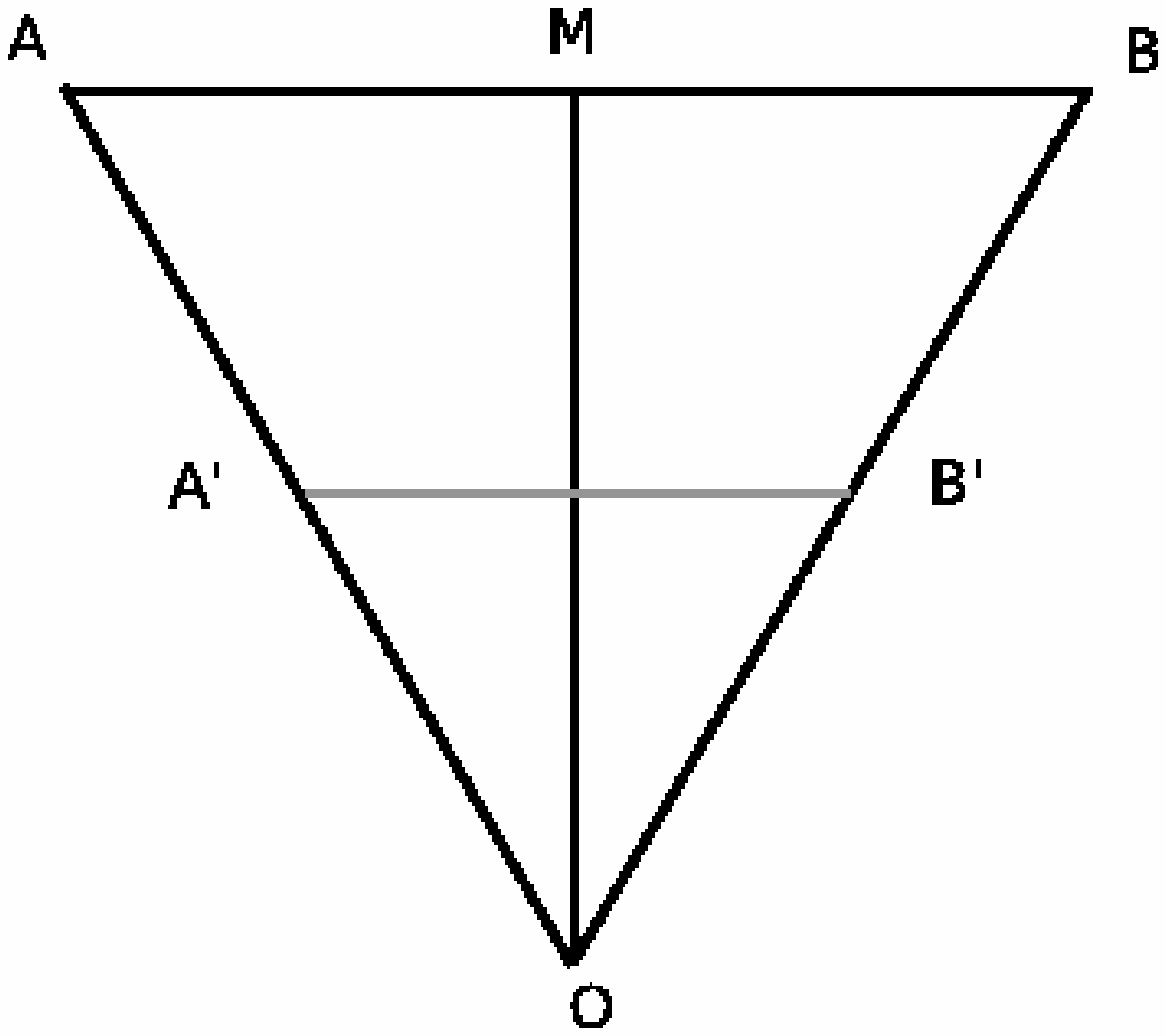}
\includegraphics[width=5cm]{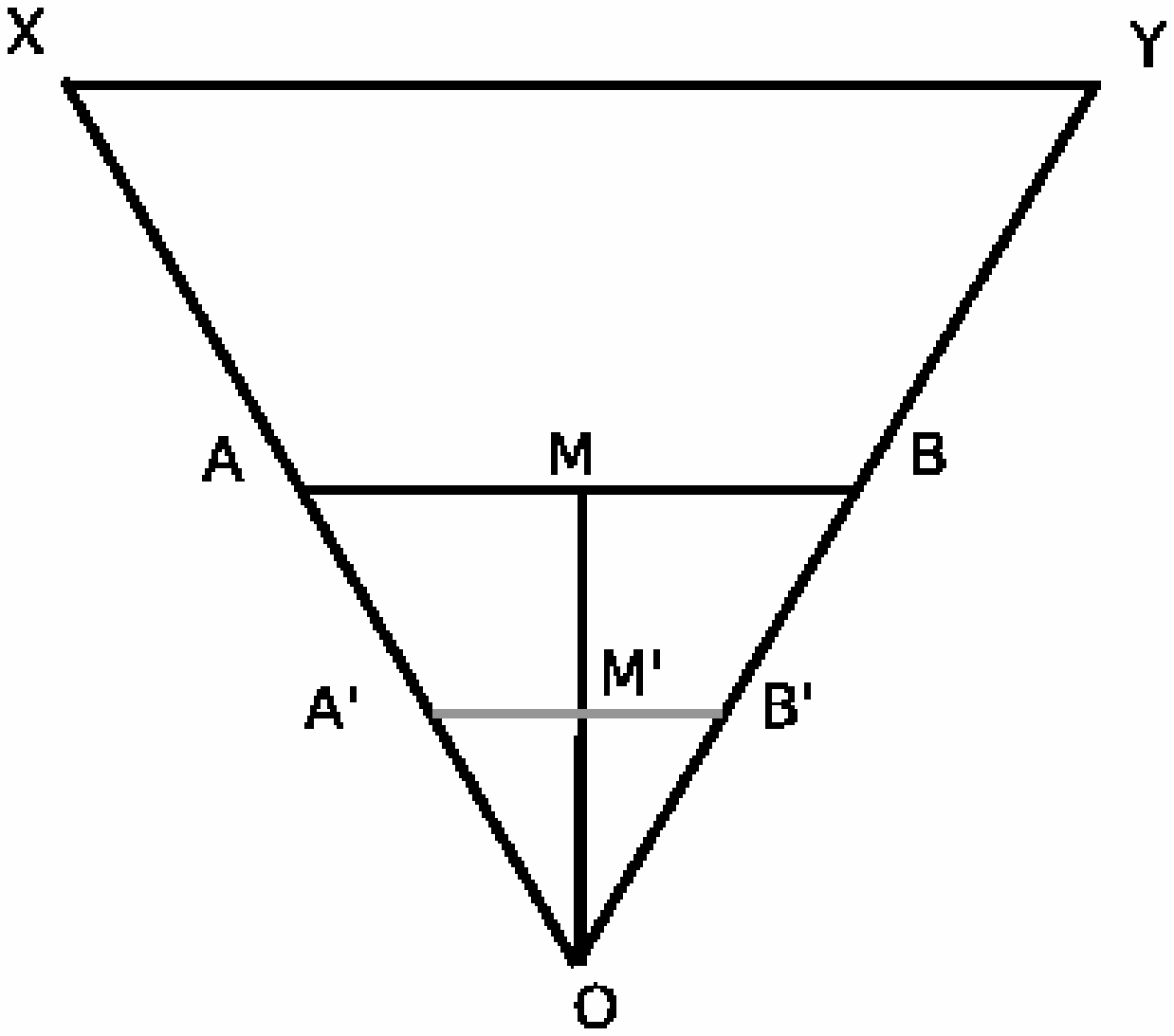}
\caption{\small Geometry of (left) a continuous piecewise affine
map with $h_\mult(T,P)>h_\top(T)=0$; (right) a discontinuous
piecewise affine map with no maximum measure.}\label{fig-ex-d2}
\end{figure}

Consider some triangle $ABO$ in $\RR^2$ with non-empty interior
and let $M,A',B'$ be the middle points of $[AB],[AO],[BO]$ (see
Fig. \ref{fig-ex-d2}, left). Let $T$ be affine in each of the
triangles $\tau_0:=AMO$ and $\tau_1:=BMO$ with $T(O)=O$,
$T(A)=T(B)=A'$, $T(M)=B'$ so that $T:ABO\to ABO$ is conjugate to
$(\theta,r)\mapsto(1-2|\theta|,r/2)$ on $(-1,1)\times(0,1)$. Take
$P=\{\tau_0,\tau_1\}$ as the admissible partition.

$h_\mult(T,P)=\log 2$ (because all words on $\{\tau_0,\tau_1\}$
appear in the symbolic dynamics and the corresponding cylinders contain
$O$ in their closure). On the other hand, the only
invariant probability measure is the Dirac supported by $O$ hence
$h_\top(T)=0$ as claimed.

\begin{example}[{\bf see Kruglikov and Rypdal
\cite{KR}}]\label{ex-pw-d3-multpositive} There exists a piecewise
affine homeomorphism $(M,T,P)$ with $\dim M=3$ and
$h_\mult(T,P)=h_\top(T)>0$.
\end{example}

Let $([0,1]^2,T_2,P_2)$ be a piecewise affine homeomorphism with
nonzero topological entropy. Consider the pyramid
$M:=\widehat{[0,1]^2}$ where $\widehat{A}$ denotes the convex subset of $\RR^3$
generated by $O:=(0,0,0)$ and $A\times\{1\}$. Define $T:M\to M$ as
the piecewise affine map with partition $P:=\{\widehat{A}\setminus\{O\}:A\in
P_2\}$ such that: $T(O)=O$, $T(x,y,1)=(T_2(x,y),1)$ for
each vertex $(x,y)$ of $P_2$. Observe that $h_\top(T)=h_\top(T_2)$
and that $T$ has an obvious measure of maximal entropy carried by
the invariant set $[0,1]^2\times\{1\}$. Finally, considering $T^n$
around $(0,0,0)$ it is easy to see that $h_\mult(T,P)=h_\top(T_2)
=h_\top(T)$.

\begin{example}\label{ex-pw-d3-mult}
There exists a piecewise affine homeomorphism $(M,S,P)$ with $\dim
M=3$, $h_\mult(S,P)>0$ and $h_\top(S)=0$.
\end{example}

Define $S$ from the previous example $T$ by $S(x,y,z):=
T(x,y,z)/2$ on the pyramid $M$ so that $0$ is a sink. To make $S$
onto, add a symmetric pyramid $M^-$ whose summit is a source.

\medbreak\noindent{\bf No Maximal Entropy Measure}

\begin{example}\label{ex-pw-d2-max}
There exists a piecewise affine surface $(M,T,P)$
\emph{discontinuous map} $T$ such that there is no maximum measure. More precisely
there exists a sequence of invariant probability measures $\mu_n$ with
 $$
  \lim_{n\to\infty} h(T,\mu_n)=h_\top(T)>0
 $$
but $\mu_n$ converges weakly to an invariant Dirac measure.
\end{example}

\medbreak\noindent{\bf Remark.} \emph{The above formulation excludes trivial
examples like $T:[0,1]\to[0,1]$ with $T(x)=1/4+x/2$ for $x>1/2$
and $T(x)=x+1/2$ for $x\leq 1/2$ which has no invariant
probability measure.}

\medbreak

Let $T$ be a piecewise affine map defined on the triangle $XYO$
with $O=(0,0)$, $X=(-2,2)$ and $Y=(2,2)$. Let $A=(-1,1)$,
$B=(1,1)$ and $M=(0,1)$, and $A'=A/2$, $B'=B/2$ and $M'=M/2$ (see
Fig. \ref{fig-ex-d2}, right). We require:
 \begin{enumerate}
  \item $T|\overline{XYBA}$ is the identity;
  \item $T:AMO\to A'B'O$ is affine with $A\mapsto A'$, $M\mapsto B'$, $O\mapsto O$;
  \item $T:MBO\to YXO$ is affine with $M\mapsto Y$, $B\mapsto X$, $O\mapsto
  O$.
 \end{enumerate}
It is easy to see that $h_\top(T)=\log2$. We claim that $\sup_\mu h(T,\mu)=\log 2$.
Clearly the supremum is bounded by $h_\top(f)$. Conversely, for any $h<\log
2$, one can find an invariant measure on the full shift $(\sigma,\{0,1\}^\NN)$ such that
$\mu([1^K])=0$ for some $K=K(h)<\infty$ with $h(\sigma,\mu)>h$. It
is then easy to construct an isomorphic $T$-invariant measure
(with support included in $y\leq y_0$ for any given $0<y_0<1$), proving that $\sup_\mu
h(T,\mu)\geq\log 2$. The equality follows from $h_\top(T)=\log2$.
The same observations allow the construction of the sequence $\mu_n$
with the claimed properties.

On the other hand, assume that $\mu$ is an invariant and ergodic
probability measure with $h(T,\mu)=2$. $\mu$ must be supported on
$y <1$. Hence the map $\pi$ that sends a point of $\RR^2$ to the
ray from the origin that contains it, maps $(T,\mu)$ to
$(f,\pi_*\mu)$ where $f:\theta\mapsto 1-2|\theta|$ on
$[-1,1]$. The fibers of $\pi$ are contained in line segments originating
from $O$ on which
$T$ is linear, hence they have zero entropy and $\pi$ is
entropy-preserving \cite{BowenEntropy}. This implies that $\pi_*\mu$ is the
$(1/2,1/2)$-Bernoulli measure. Using, say the Central Limit
Theorem, we get that, for $\mu$-a.e. $(x,y)\in ABO$, there exists
a positive integer $n$ such that
 $$
    \#\{0\leq k<n:f^k(\pi(x,y))<1/2\} < \frac n2-\frac{|\log y|}{2\log
    2}
 $$
so that $T^m(x,y)\in XYBA$ for some $m\leq n$, contradicting the invariance of
the measure: there
is no maximum measure.

\begin{example}\label{ex-C0-q-d2-max}
There exists a continuous, piecewise \emph{quadratic} surface map
$T$ such that for any invariant probability measure $\mu$:
 $$
    h(T,\mu)<\sup_\nu h(T,\nu).
 $$
\end{example}

On the rectangle $[1,2]\times[-1,1]$, consider
$T(x,y):=(x,T_x(y))$ with:
 $$
   T_x(y)=\left\{ \begin{matrix} \frac{x(2-x)}2-x|y| &\text{if
   }|y|<2-x \\
   -\frac{x(2-x)}2\qquad & \text{otherwise} \end{matrix}\right.
 $$
For each $1\leq x<2$,  $[-1,1]$ is mapped into the $T_x$-forward
invariant segment $[-\frac{x(2-x)}2,\frac{x(2-x)}2]$ on which
$T_x$ has constant slope $x$. Hence $h_\top(T_x)=\log x$ for $x\ne
2$. Clearly, $T(2,y)=(2,0)$ so $h_\top(T_2)=0$.

}

\medbreak\noindent{\bf Infinitely Many Maximal Entropy Measures}

\begin{example}
There is a piecewise affine continuous map, resp. homeomorphism, 
of $[0,1]^2$, resp. $[0,1]^3$, with uncountably many ergodic 
invariant probability measures with non-zero, maximal entropy.
\end{example}

Indeed, such examples are trivially obtained from 
piecewise affine maps on $[0,1]$ or homeomorphisms
on $[0,1]^2$ with non-zero 
topological entropy by taking a direct product with 
the identity on the unit interval. It is the \emph{low
dimension} (one for maps, two for homeomorphisms) that
prevents the existence of such indifferent factors and
ensures the finite number of maximum measures  under
the simple condition of non-zero topological entropy.

\end{document}